\documentclass[a4paper,10pt,english,leqno]{amsart}
\usepackage{color}
\usepackage{dsfont}
\usepackage{mathrsfs}

\usepackage[colorlinks=false]{hyperref}

\numberwithin{equation}{section}
\allowdisplaybreaks

\usepackage{environ}
\makeatletter
\NewEnviron{Lalign}{\tagsleft@true\begin{align}\BODY\end{align}}
\makeatother

\newcommand{\R}{\mathbb{R}}
\newcommand{\Z}{\mathbb{Z}}
\newcommand{\test}{\phi}
\newcommand{\seq}[1]{\left\{#1\right\}}
\newcommand{\abs}[1]{\left|#1\right|}
\newcommand{\eps}{\varepsilon}
\DeclareMathOperator*{\sgn}{sign}
\newcommand{\sign}[1]{\sgn\left(#1\right)}
\newcommand{\signe}[1]{\mathrm{sign}_\varepsilon\left(#1\right)}
\newcommand{\norm}[1]{\left\|#1\right\|}
\newcommand{\car}[1]{\mathds{1}_{#1}}

\newcommand{\Dm}{D_-}
\newcommand{\Dp}{D_+}
\newcommand{\Dx}{{\Delta x}}

\newcommand{\A}{\mathcal{A}}
\newcommand{\conv}[1]{\stackrel{(#1)}{\star}}
\newcommand{\Lip}{\mathrm{Lip}}
\newcommand{\Qeps}{Q_{\eps,r,r_0}}
\newcommand{\Eeps}{\mathcal{E}_{\eps,r,r_0}^{\Dx}}
\newcommand{\BEeps}{\overline{\mathcal{E}}_{\eps,r,r_0}^{\Dx,\tau}}
\newcommand{\term}{\mathscr{T}}

\theoremstyle{remark}\newtheorem{remark}{Remark}[section]

\theoremstyle{plain}\newtheorem{lemma}{Lemma}[section]

\newtheorem{estimate}{Estimate}[section]
\newtheorem{theorem}{Theorem}[section]
\theoremstyle{definition}\newtheorem{definition}{Definition}[section]

\title[Convergence rate for degenerate parabolic equations]
{On the convergence rate of finite difference methods for degenerate  
convection-diffusion equations in several space dimensions} 

\author[K. H. Karlsen]{K. H. Karlsen} \address[Kenneth Hvistendahl Karlsen]
{Department of Mathematics\\
University of Oslo\\
P.O.~Box 1053, Blindern\\
N-0316 Oslo, Norway.}
\email[]{kennethk@math.uio.no}

\author[N. H. Risebro]{N. H. Risebro} \address[Nils Henrik Risebro] 
{Department of Mathematics\\
University of Oslo\\
P.O.~Box 1053, Blindern\\
N-0316 Oslo, Norway.}
\email[]{nilshr@math.uio.no}

\author[E. B. Storr\o{}sten]{E. B. Storr\o{}sten} \address[Erlend
Briseid Storr\o{}sten] 
{Department of Mathematics\\
University of Oslo\\
P.O.~Box 1053, Blindern\\
N-0316 Oslo, Norway.}
\email[]{erlenbs@math.uio.no}

\date{\today}



\begin{document}

\begin{abstract}
We analyze upwind difference methods for strongly degenerate 
convection-diffusion equations in several spatial dimensions. 
We prove that the local $L^1$-error between the exact and numerical solutions 
is $\mathcal{O}\bigl(\Dx^{2/(19+d)}\bigr)$, where $d$ is the
spatial dimension and $\Dx$ is the grid size.  The 
error estimate is robust with respect to vanishing 
diffusion effects. The proof makes effective use of specific kinetic
formulations of the difference method and the convection-diffusion equation. 
This paper is a continuation of \cite{KRS2014}, in which 
the one-dimensional case was examined using 
the Kru{\v{z}}kov-Carrillo entropy framework.
\end{abstract}

\maketitle

\tableofcontents

\section{Introduction}
The design of numerical methods for convection-diffusion 
problems is important for many applications in science and engineering. 
It is especially challenging to construct 
accurate methods for nonlinear problems in 
which the ``diffusion part" is small or vanishing, relative to 
the ``convection part" of the problem. Connected to this is the  
difficult problem of deriving error estimates for numerical methods 
that are robust in the singular limit as the diffusion coefficient 
vanishes, thereby avoiding the usual exponential growth of error constants. 

In this paper we are interested in deriving error estimates for a class of 
finite difference methods for nonlinear, possibly strongly 
degenerate, convection-diffusion problems of the form
\begin{equation}
  \begin{cases}\label{eq:CD}
    \partial_tu + \nabla \cdot f(u) = \Delta A(u), & (t,x) \in \Pi_T, \\
    u(0,x) = u_0(x), & x \in \R^d,
  \end{cases}
\end{equation}
where $\Pi_T = (0,T) \times \R^d$, $T>0$, $d\ge 1$, 
and $u:\Pi_T \rightarrow \R$ is the
unknown function that is sought. The initial datum $u_0$ is
an integrable and bounded function, while the \emph{flux} 
function $f:\R \rightarrow \R^d$ and the \emph{diffusion} 
function $A:\R\rightarrow \R$ satisfy the standing assumptions
\begin{equation}\label{assump:fluxAndDiff}
	\text{$f,A$ locally $C^1$; $A(0) = 0$; $A$ is nondecreasing.}
\end{equation}

By \emph{strongly degenerate} it is meant that we allow for $A'(u) = 0$ for
all $u$ in some interval $[\alpha,\beta] \subset \mathbb{R}$. The
resulting class of equations therefore contains parabolic and 
hyperbolic equations, as well as a mix thereof. In the 
nondegenerate (uniformly parabolic) case $A'(\cdot)>0$, it is 
well known that \eqref{eq:CD} admits a unique classical solution. 
On the other hand, for strongly degenerate equations 
with discontinuous solutions, the well-posedness is 
ensured only in a class of weak solutions 
satisfying an entropy condition. The following result is 
known: For $u_0\in L^1(\R^d)\cap L^\infty(\R^d)$, there
exists a unique solution $u\in C((0,T);L^1(\R^d))$, $u\in
L^\infty(\Pi_T)$ of \eqref{eq:CD} 
such that $\partial_x A(u)\in L^2(\Pi_T)$ and for all 
convex functions $S$ with 
$q_S'=S' f'$ and $r_S'=S'A'$,
\begin{equation*}
	\partial_t S(u) + \nabla\cdot q_S(u) 
	- \Delta r_S(u)\le 0, \quad 
	\text{weakly on $[0,T)\times \R^d$}.
\end{equation*}
These inequalities are referred to as entropy 
inequalities and the corresponding 
solution is called an entropy solution.  

For conservation laws ($A'\equiv 0$), the well-posedness of
entropy solutions is a celebrated result due to Kru{\v{z}}kov
\cite{Kruzkov:1970kx}. Carrillo \cite{Carrillo:1999hq} extended this result to
degenerate parabolic problems such as \eqref{eq:CD}.  
For uniqueness of entropy solutions in the $BV$ class, 
see \cite{VolpertHudjaev1969,Wu:1989ve}.
An alternative well-posedness theory, based on the 
so-called kinetic formulation, was developed 
by Lions, Perthame, and Tadmor \cite{Lions:1994qy} and 
Chen and Perthame \cite{ChenPerthame2003}. 
We refer to \cite{Andreianov:2012fk,Dafermos:2010fk} 
for an overview of the relevant literature 
on hyperbolic and mixed hyperbolic-parabolic problems. 

In this paper we derive error estimates for 
numerical approximations of entropy
solutions to convection-diffusion equations. Convergence results, 
without error estimates, have been obtained for difference
methods \cite{EvjeKarlsen2000,Evje:2000ix,Karlsen:2001ul}; finite volume methods
\cite{Eymard:2002nx,Andreianov:2009kx}; splitting 
methods \cite{Holden:2010fk}; and BGK approximations
\cite{Aregba-Driollet:2003, Bouchut:2000dp}, to mention a few references. 
For a posteriori error estimates for finite volume methods, see
\cite{Ohlberger:2001oq}.

We are herein interested in estimating the error committed by a class of 
monotone difference methods. The monotone methods 
make use of an upwind discretization of the convection term and 
a centred discretization of the parabolic term. 
For notational simplicity in the introduction, let us 
assume $f^{i,\prime}(\cdot)\ge 0$ and consider the prototype
(semi-discrete) difference method
\begin{align*}
  \frac{d}{dt} u_\alpha 
  & + \sum_{i=1}^d 
  \frac{f^i(u_\alpha)-f^i(u_{\alpha-e_i})}{\Dx} 
  = \sum_{i=1}^d
  \frac{A(u_{\alpha+e_i})-2A(u_\alpha)+A(u_{\alpha-e_i})}{\Dx^2},
\end{align*}
where $\alpha=(\alpha_1,\dots,\alpha_d)\in \Z^d$ is a multi-index, $e_i$ is the $i$th unit 
vector in $\R^d$, and $\Dx > 0$ is the spatial grid size. Although our methods 
are semi-discrete, i.e., not discretized in time, 
the results and proofs can be adjusted to cover some 
fully discrete methods, such as the implicit method analyzed 
in \cite{EvjeKarlsen2000}. We refer to \cite{KRS2014} 
for a discussion of this topic when $d=1$.

Denote by $u_\Dx$ the piecewise constant 
interpolant linked to $u_\alpha$.   
The goal is to determine a number 
(convergence rate) $\gamma>0$ such that 
\begin{equation}\label{eq:rate1}
  \norm{u_\Dx(t,\cdot)-u(t,\cdot)}_{L^1} 
  \le C \Dx^\gamma, 
  \qquad \text{($u_0\in BV$)},
\end{equation}
for some constant $C>0$ independent 
of $\Dx$ and (the smallness of) $A'$. 

In the purely hyperbolic case ($A'\equiv 0$), a prominent result due
to Kuznetsov \cite{Kuznetsov} says that $\gamma$ is 
$1/2$ for monotone difference methods, as well as 
for the vanishing viscosity method. Influenced 
by \cite{Kuznetsov}, a number of works have 
further developed the ``Kru{\v{z}}kov-Kuznetsov"  error 
estimation theory for conservation laws.  We refer to
\cite{Bouchut:1998ys,Cockburn:2003ys} for an overview of the relevant results. 
Making use of the kinetic formulation, Perthame \cite{Perthame1998} 
provided an alternative route to error estimates.

With regard to convection-diffusion equations \eqref{eq:CD} 
with $A'(\cdot)\ge 0$, the subject of error estimates 
is significantly more difficult. It is only recently 
that there has been some progress. 
The simplest case is the vanishing viscosity method.
Denote by $u^\eta$ the solution of the uniformly parabolic equation 
\begin{equation}\label{intro:vv}
	u^\eta_t + \nabla \cdot f(u^\eta) 
	= \Delta A^\eta(u^\eta),  
	\qquad A^\eta(u) = A(u) + \eta u,
\end{equation}
where $\eta > 0$ is a (small) viscosity parameter.  
We have the following error estimate 
for the viscosity approximation $u^\eta$:
\begin{equation}\label{eq:VV-error}
	\norm{u^\eta(\cdot,t)-u(\cdot,t)}_{L^1}
	\le C \, \sqrt{\eta}, \qquad \text{($u_0\in BV$)},
\end{equation}
where $u$ is the entropy solution of \eqref{eq:CD}.
A ``Kru{\v{z}}kov-Kuznetsov"  type proof of this result is given 
in \cite{EvjeKarlsen2002}, see also \cite{Eymard:2002eu} 
for a boundary value problem. 
The error bound \eqref{eq:VV-error} can also be seen as an 
outcome of continuous dependence 
estimates \cite{Cockburn:2003ys,Chen:2005wf} or 
the kinetic formulation \cite{Chen:2006oy,Makridakis:2003aa}.

For conservation laws, the error estimate \eqref{eq:VV-error} 
for the viscous equation reveals what to expect for 
monotone difference methods \cite{Kuznetsov}.
This suggestive link breaks down for degenerate convection-diffusion 
equations \eqref{eq:CD}, cf.~\cite{KRS2014}, a fact that may foreshadow 
added difficulties coming from a second order 
operator. Indeed, for general $A$ satisfying \eqref{assump:fluxAndDiff} and 
in one spatial dimension, the work \cite{KKR2012} established \eqref{eq:rate1} 
with $\gamma=1/11$, a rate that was recently improved 
to $\gamma=1/3$ in \cite{KRS2014}. Although $\gamma=1/3$ is the 
best available rate at the moment, its optimality is unknown and also 
far from the convergence rate $\gamma=1/2$ known to be optimal 
for conservation laws. But in spite of that, with a linear diffusion 
function $A$, the convergence rate improves 
to $\gamma=1/2$ \cite{KKR2012,KRS2014}.

Apart from a result ($\gamma=1/2$) for linear convection-diffusion 
equations \cite{Chen:2006oy}, we are not aware of any results 
for multi-dimensional  equations \eqref{eq:CD} with a 
degenerate, nonlinear diffusion part. 
In this paper we establish \eqref{eq:rate1} with 
\begin{equation}\label{eq:rate2}
	\gamma=\frac{2}{19+d} 
	\quad \text{($d$ is the spatial dimension)},
\end{equation}
for general diffusion functions $A$ obeying \eqref{assump:fluxAndDiff}. 

A technical aspect of the proof of \eqref{eq:rate1} is that 
we are not applying the difference method 
directly to \eqref{eq:CD} but rather to \eqref{intro:vv}. 
Denoting the corresponding numerical solution 
by $u^\eta_\Dx$, we will prove that \eqref{eq:rate1} holds 
with $u_\Dx, u$ replaced by $u^\eta_\Dx,u^\eta$, respectively, and that 
the error constant $C$ is not depending on the parameter $\eta$. 
Our original claim \eqref{eq:rate1} follows from this, since 
we have the error estimate \eqref{eq:VV-error}. 

To help motivate the technical arguments coming later, let us 
lay out a ``high-level" overview of the analysis and 
some of the difficulties involved. As just alluded to, we will mostly work 
under the assumption $A' > 0$.  As a consequence no information is lost 
upon working with $A(u)$ instead of $u$ in the 
kinetic formulation (compare with the $u$-based 
formulation in \cite{ChenPerthame2003}).  
Set $B = A^{-1}$ and define $g$ by $g \circ A = f$. 
Then the solution $u$ of \eqref{eq:CD} satisfies
\begin{equation}\label{eq:KineticFormulation}
 B'(\zeta) \partial_t \chi_{A(u)} 
 + g'(\zeta)\cdot \nabla \chi_{A(u)} 
 -\Delta \chi_{A(u)}  = \partial_\zeta m_{A(u)},
\end{equation}
where
\begin{align*}
  &m_{A(u)} =m_{A(u)} (\zeta) 
  =\delta(\zeta-A(u))\abs{\nabla A(u)}^2,
  \\ & 
  \chi_{A(u)}=\chi_{A(u)}(\zeta) 
  = \begin{cases}
    1 & \mbox{ if $0 < \zeta \le A(u)$,} \\
    -1 & \mbox{ if $A(u) \le \zeta < 0$,} \\
    0 & \mbox{ otherwise.}
  \end{cases}
\end{align*}
This new formulation, although restricted to 
nondegenerate (isotropic) diffusion, allows 
for a simpler proof of the $L^1$ contraction property and 
thus the error estimate \eqref{eq:rate1}. More specifically, certain error 
terms linked to the regularization of the $\chi$ function \cite{ChenPerthame2003} 
can be avoided, a fact that we use to our benefit.

Now we indicate how \eqref{eq:KineticFormulation} leads 
to the $L^1$ contraction property. 
Let $u$ and $v$ be solutions to \eqref{eq:CD} with initial 
values $u_0$ and $v_0$, respectively. 
Following \cite{ChenPerthame2003,Perthame1998}, we 
introduce the microscopic contraction functional
\begin{equation}\label{eq:QDef}
	Q_{u,v}(\xi) = \abs{\chi_u(\xi)} + \abs{\chi_v(\xi)} - 2\chi_u(\xi)\chi_v(\xi).
\end{equation}
Under the change of variable $\zeta = A(\xi)$,
\begin{equation*}
	\abs{u-v} = \int_\R Q_{u,v}(\xi)\,d\xi 
	= \int_\R B'(\zeta) Q_{A(u),A(v)}(\zeta)\,d\zeta,
\end{equation*}
and hence
\begin{equation}\label{eq:TimeDerL1}
	\begin{split}
		\partial_t\abs{u-v} 
		& = \int_\R B'(\zeta)  \partial_t Q_{A(u),A(v)}(\zeta) \,d\zeta \\
		& = \int_\R B'(\zeta) \partial_t \abs{\chi_{A(u)}(\zeta)}\,d\zeta
		+   \int_\R B'(\zeta) \partial_t \abs{\chi_{A(v)}(\zeta)}\,d\zeta \\
		& \qquad -2 \int_\R B'(\zeta) 
		\partial_t\Bigl(\chi_{A(u)}(\zeta)\chi_{A(v)}(\zeta)\Bigr)\,d\zeta.
  \end{split}
\end{equation}

In view of \eqref{eq:KineticFormulation}, the chain rule yields 
\begin{equation}\label{eq:chiAbsTimeDer}
  B'(\zeta) \partial_t \abs{\chi_{A(u)}} 
  + g'(\zeta) \cdot \nabla \abs{\chi_{(A(u)}}
  - \Delta \abs{\chi_{A(u)}} 
  = \sign{\zeta}\partial_\zeta m_{A(u)},
\end{equation}
with an analogous equation for $v$. 
Using the equations for $\chi_{A(u)}, \chi_{A(v)}$ and 
Leibniz's product rule, we easily check that 
\begin{equation}\label{eq:chiProdTimeDerNew}
  \begin{aligned}
    & B'(\zeta) \partial_t \Bigl (\chi_{A(u)}  \chi_{A(v)}\Bigr) +
    g'(\zeta) \cdot \nabla \Bigl (\chi_{A(u)}  \chi_{A(v)}\Bigr) 
    - \Delta \Bigl (\chi_{A(u)}  \chi_{A(v)}\Bigr) \\ 
    &\qquad\qquad = \chi_{A(v)}\partial_\zeta m_{A(u)} 
    + \chi_{A(u)}\partial_\zeta m_{A(v)} 
    - 2 \nabla \chi_{A(u)} \cdot \nabla \chi_{A(v)}.
  \end{aligned}
\end{equation}
Making use of \eqref{eq:chiAbsTimeDer} and \eqref{eq:chiProdTimeDerNew} 
in \eqref{eq:TimeDerL1} yields 
\begin{equation}\label{eq:FormalContractioneEquation}
\begin{split}
    \partial_t\abs{u-v}   
    = - \int_\R & g'(\zeta) \cdot \nabla Q_{A(u),A(v)}(\zeta)\,d\zeta
    +\int_\R \Delta Q_{A(u),A(v)}(\zeta)\,d\zeta 
    \\ & + \int_\R D(\zeta)\,d\zeta,
\end{split}
\end{equation}
where
\begin{align*}
  D(\zeta) & =
  \Bigl(\sign{\zeta}-2\chi_{A(v)}(\zeta)\Bigr)\partial_\zeta m_{A(u)}
  + \Bigl(\sign{\zeta}-2\chi_{A(u)}(\zeta)\Bigr) \partial_\zeta
  m_{A(v)} \\ & \qquad \qquad + 4\nabla \chi_{A(u)}(\zeta) \cdot
  \nabla \chi_{A(v)}(\zeta)
  \\
  &=: D_1(\zeta)+D_2(\zeta)+D_3(\zeta);
\end{align*}
the term $D(\cdot)$ accounts for the parabolic 
dissipation effects associated with $u,v$. 
Integrating \eqref{eq:FormalContractioneEquation} 
in $x$ gives
$$
\frac{d}{dt} \int \abs{u(t,x)-v(t,x)}\, dx 
= \int \int_\R D(\zeta)\,d\zeta\, dx.
$$
Although the computations have been formal up to this point, they 
are valid when interpreted in the sense of distributions. Moreover, 
as will be seen later, these computations can with some effort be replicated at 
the discrete level, i.e., when we replace the function $v$ by the 
numerical solution $u_\Dx$.

Clearly, the $L^1$-contraction property follows if we can confirm that
\begin{equation}\label{eq:SignOfDissipTerm}
	\int_\R D(\zeta) \,d\zeta \le 0.
\end{equation}
This step is rather delicate and will ask for a 
regularization of the $\chi$ functions. Indeed, the hard part 
of the proof leading up to \eqref{eq:rate1}, \eqref{eq:rate2} 
is related to this step. Let us for the moment ignore the regularization 
procedure, and continue with formal computations. 
Note that
\begin{equation*}
	\sign{\zeta}-2\chi_{A(v)}(\zeta) = \sign{\zeta-A(v)},
\end{equation*}
and thus, after an integration by parts 
followed by an application of the chain rule,
\begin{align*}
  &\int_\R D_1(\zeta)  \,d\zeta 
  = -2\int_\R \delta(\zeta-A(u)) 
  \delta(\zeta-A(v))\abs{\nabla A(u)}^2\,d\zeta.  
\end{align*}
Similarly,
\begin{align*}
  &\int_\R D_2(\zeta)  \,d\zeta 
  = -2\int_\R \delta(\zeta-A(u)) 
  \delta(\zeta-A(v))\abs{\nabla A(v)}^2\,d\zeta.  
\end{align*}
Again by the chain rule,
\begin{equation*}
  D_3(\zeta)= 4\delta(\zeta-A(u))\delta(\zeta-A(v))
  \nabla A(u) \cdot \nabla A(v). 
\end{equation*}
Combining these formal computations we finally 
arrive at \eqref{eq:SignOfDissipTerm}:
\begin{equation*}
  \int_\R D(\zeta)\,d\zeta 
  = -2\int_\R \delta(\zeta-A(u))\delta(\zeta-A(v))
  \abs{\nabla A(u)-\nabla  A(v)}^2\,d\zeta \leq 0.
\end{equation*}

One crucial insight in \cite{KRS2014}  
is that the convergence rate can be improved if one can send
a certain parameter $\eps$ to zero independently of the grid size
$\Dx$, where $\eps$ controls  the regularization of 
the Kru\v{z}kov entropies. In this paper the 
regularization of the entropies is replaced by the regularization 
of the $\chi$ functions, and as before we would like to 
send $\eps$ to zero independently of $\Dx$ (and other parameters).  
It turns out that in one spatial dimension we can do this, reaching the convergence 
rate $\gamma=1/3$ as in \cite{KRS2014}. In several dimensions we 
have not been able to carry out this ``$\eps\to 0$ before 
other parameters" program.  

A serious difficulty stems from the lack of a chain rule for 
finite differences, in combination with the highly nonlinear nature 
of the dissipation function $D(\cdot)$, resulting in a 
series of intricate error terms. A feature of the kinetic approach is that 
the crucial error term can be expressed via the 
parabolic dissipation measure.  To be a bit more 
precise, at the continuous level, the 
convergence rate $\gamma=1/3$ in the one-dimensional 
case depends decisively on the (weak) continuity of the map
\begin{equation}\label{eq:pardiss}
	c \mapsto \int_\R \delta(\zeta-c)m_{A(u)}(\zeta)\,d\zeta 
	= \delta(c-A(u))(\partial_xA(u))^2,
\end{equation}
where $u$ is the entropy solution and $m_{A(u)}$ is the parabolic 
dissipation measure.  The continuity of this map 
follows from \eqref{eq:KineticFormulation}. Unfortunately, in 
several space dimensions the continuity becomes 
a subtle matter, since the parabolic dissipation 
measure splits into directional components,
$$
m_{A(u)} = \sum_{i = 1}^d m_{A(u)}^i, \qquad
m_{A(u)}^i = \delta(\zeta-A(u))(\partial_{x_i}A(u))^2.
$$ 
It appears difficult to claim from 
the kinetic equation \eqref{eq:KineticFormulation}
the continuity of the individual components
$$
c \mapsto \int_\R \delta(\zeta-c)m_{A(u)}^i(\zeta)\,d\zeta 
= \delta(c-A(u))(\partial_{x_i}A(u))^2, \quad i=1,\ldots, d.
$$
Not being able to send the $\chi$-regularisation 
parameter $\eps$ to zero, we must instead 
balance $\eps$ against the grid size $\Dx$ and a number 
of other parameters, at long last arriving at \eqref{eq:rate1} with the 
convergence rate \eqref{eq:rate2}. 

The optimality of \eqref{eq:rate2} ($d>1$), in 
the $L^\infty\cap BV$  class, is an open problem. 
It is informative to compare with recent results on viscosity solutions 
and error bounds for degenerate fully nonlinear elliptic and parabolic equations. We 
refer to Krylov \cite{Krylov:2005lj}, Barles and Jakobsen \cite{Barles:2005sc}, and 
Caffarelli and Souganidis \cite{Caffarelli:2008aa} for some recent works. 
For monotone approximations of fully nonlinear, first-order equations 
with Lipschitz solutions, Crandall and Lions \cite{CranLions:FDM84} proved in 1984 the 
optimal $L^\infty$ convergence rate $1/2$. However, finding a 
rate for degenerate second order equations remained an open problem.
The first result is due to Krylov with the rate $1/27$. 
Later Barles and Jakobsen improved this to to $1/5$, with 
a further improvent by Krylov to $1/2$ for equations with special structure.
We remark that these results concern equations with convex nonlinearities.
Caffarelli and Souganidis proved that there is an algebraic rate of convergence for 
a class of nonconvex equations. The convergence rate is not explicit 
but known to be some (small) positive number. 
Here we should point out that in our framework convexity plays no role; the 
error estimate applies to general nonlinearities.

The remaining part of this paper is organized as follows:  
In Section \ref{sec:viscappentropy} we gather 
some relevant a priori estimates for 
nondegenerate convection-diffusion equations 
and state precisely the definition of an entropy solution.  The 
difference method and the main result are presented in 
Section \ref{sec:differencescheme}. In Section~\ref{sec:kinetic} we supply 
certain kinetic formulations of the convection-diffusion 
and difference equations. Section~\ref{sec:errest} is 
devoted to the proof of the main result, achieved through the derivation 
of an error equation based on the kinetic formulations, along 
with a lengthy series of estimates bounding ``unwanted" terms in this equation. 
In Appendix~\ref{app:semiexist} we collect results 
relating to well-posedness and a priori estimates for the difference method.

\section{Viscosity approximations and entropy solutions}\label{sec:viscappentropy}
Let us define the viscosity approximations. 
Set $A^\eta(u) := A(u) + \eta u$ for any fixed $\eta > 0$, and
consider the the uniformly parabolic problem
\begin{equation}\label{eq:viscousApproximation}
  \begin{cases}
    u^\eta_t + \nabla \cdot f(u^\eta) = \Delta A^\eta(u^\eta), & (t,x) \in \Pi_T,\\
    u^\eta(0,x) = u_0(x), & x \in \R^d.
  \end{cases}
\end{equation}
It is well known that \eqref{eq:viscousApproximation} admits a unique
classical (smooth) solution.  We collect some relevant (standard)
estimates from \cite{VolpertHudjaev1969}.

\begin{lemma}\label{lem:1.1}
  Suppose $u_0\in L^\infty(\R^d)\cap L^1(\R^d) \cap BV(\R^d)$, and let
  $u^\eta$ be the unique classical solution of
  \eqref{eq:viscousApproximation}. Then for any $t>0$,
  \begin{align*}
    \norm{u^\eta(t,\cdot)}_{L^1(\R^d)} & \le \norm{u_0}_{L^1(\R^d)},\\
    \norm{u^\eta(t,\cdot)}_{L^\infty(\R^d)} &\le \|u_0\|_{L^\infty(\R^d)}, \\
    \abs{u^\eta(t,\cdot)}_{BV(\R^d)} &\le \abs{u_0}_{BV(\R^d)}.
  \end{align*}
\end{lemma}

\begin{lemma}\label{LipschitzContTime}
  Suppose $u_0\in L^\infty(\R^d)\cap L^1(\R^d)$ and $\nabla \cdot
  (f(u_0)-\nabla A(u_0)) \in L^1(\R^d)$. Let $u^\eta$ be the unique
  classical solution of \eqref{eq:viscousApproximation}.  Then for any
  $t_1,t_2>0$,
  \begin{equation*}
    \norm{u^\eta(t_2,\cdot)-u^\eta(t_1,\cdot)}_{L^1(\R)} 
    \le \norm{\nabla \cdot (f(u_0)-\nabla A(u_0))}_{L^1(\R^d)}\abs{t_2-t_1}.
  \end{equation*}
\end{lemma}
These results imply that the family $\seq{u^\eta}_{\eta > 0}$ is
relatively compact in $C([0,T];L^1_{loc}(\R^d))$.  If $u=\lim_{\eta\to
  0} u^\eta$, then
\begin{equation}\label{eq:uminusueta}
  \norm{u^\eta-u}_{L^1(\Pi_T)} \le C\eta^{1/2},
\end{equation}
for some constant $C$ which does not depend on $\eta$, see, e.g., \cite{EvjeKarlsen2002}. 
Moreover, $u$ is an entropy solution according
to the following definition:
\begin{definition}
  An \emph{entropy solution} of
  \eqref{eq:CD} is a measurable
  function $u = u(t,x)$ satisfying:
  \begin{itemize}
  \item[(D.1)] $u \in L^\infty([0,T];L^1(\R^d))\cap L^\infty(\Pi_T) \cap
    C((0,T);L^1(\mathbb{R}^d))$.
  \item[(D.2)] $A(u) \in L^2([0,T];H^1(\R^d))$.
  \item[(D.3)] For all constants $c \in \mathbb{R}$ and all
   test functions $0\le \phi \in C_0^\infty(\mathbb{R}^d
    \times [0,T))$, the following entropy inequality holds:
    \begin{multline*}
      \iint_{\Pi_T}|u-c|\partial_t \phi +
      \sign{u-c}(f(u)-f(c))\cdot\nabla \varphi + |A(u)-A(c)|\Delta
      \varphi \,dtdx 
      \\
      + \int_{\mathbb{R}^d}|u_0-c|\varphi(x,0)\,dx \ge 0.
    \end{multline*}
  \end{itemize}
\end{definition}
The uniqueness of entropy solutions is proved in \cite{Carrillo:1999hq}, see the introduction 
for additional references.

\section{Difference method and main result}
\label{sec:differencescheme}
Let $f=(f^1,\ldots,f^d)$, and let $\Dx$ denote the mesh size. For
simplicity we consider a uniform grid in $\R^d$ consisting of cubes
with sides $\Dx$. For a multi-index
$\alpha=(\alpha_1,\ldots,\alpha_d)\in \Z^d$, we let $I_\alpha$ denote
the grid cell
\begin{equation*}
  I_\alpha = [x_{\alpha_1 -1/2},x_{\alpha_1 + 1/2}) \times 
  \cdots \times [x_{\alpha_d-1/2},x_{\alpha_d + 1/2}),
\end{equation*}
where $x_{j+1/2}=(j+1/2)\Dx$ for $j\in \Z$.  Let $e_k\in \Z^d$ be the vector with
value one in the $k$-th component and zero otherwise. Then we define
the forward and backward discrete partial derivatives in the $k$-ht direction as
\begin{equation*}
  D^k_\pm(\sigma_\alpha) = \pm \frac{\sigma_{\alpha \pm
      e_k}-\sigma_\alpha}{\Dx} \qquad k = 1, \dots ,d. 
\end{equation*}
\begin{definition}{(Numerical flux)}\label{def:numflux}
  We call a function $F \in C^1(\mathbb{R}^2)$ a \emph{monotone two
    point numerical flux
    for $f$}, if $F(u,u) = f(u)$ and
  \begin{equation*}
    \frac{\partial}{\partial u} F(u,v) \ge 0 \quad \text{and} \quad
    \frac{\partial}{\partial v} F(u,v) \le 0 
  \end{equation*}
  holds for all $u$ and $v$. We say that the numerical flux \emph{splits}
  whenever $F$ can be written
  \begin{equation*}
    F(u,v)=F_1(u)+F_2(v).
  \end{equation*}
  Note that $F'_1\ge 0$ and $F_2'\le 0$ whenever $F$ is monotone.
\end{definition}
Let $F^k$ be a numerical flux function corresponding to $f^k$ for $k=1,\ldots,d$. 
The semi-discrete approximation of \eqref{eq:CD} is the 
solution of the equations
\begin{equation}\label{eq:Semi-discreteScheme}
  \begin{cases}
    \frac{d}{dt} u_\alpha + \sum_{i = 1}^d \Dm^i
    F^i(u_\alpha,u_{\alpha+e_i}) = \sum_{i = 1}^d \Dm^i\Dp^iA(u_\alpha), & 
      \!\!\! \alpha \in \Z^d,\; t \in (0,T), \\
    u_\alpha(0) = u_{\alpha,0}, & \!\!\! \alpha \in \mathbb{Z}^d,
  \end{cases}
\end{equation}
where $u_{\alpha,0} = \frac{1}{\Dx^d}\int_{I_\alpha} u_0(x)\, dx$. 
See Appendix~\ref{app:semiexist}, in particular Lemmas~\ref{lem:SemiDiscProp} 
and \ref{lem:semidisc_cont}, regarding
existence and solution properties to 
this infinite system of ODEs.  

Define the piecewise constant (in $x$) function $u_\Dx$ by
\begin{equation}\label{eq:PiecewiseConstantFuncDef}
  u_\Dx(t,x) = u_\alpha(t) \text{ for } x \in I_\alpha.
\end{equation}

Our main result is the  following:

\begin{theorem}\label{theorem:MultidLocal}
Suppose $f$ and $A$ satisfy \eqref{assump:fluxAndDiff} and the 
initial function $u_0$ is in $BV(\R^d)\cap L^\infty(\R^d)\cap L^1(\R^d)$. 
Let $F^i$ be a monotone, Lipschitz, two point numerical flux 
corresponding to $f^i$ that splits for $1 \leq i \leq d$. 
Let $u$ be the entropy solution to \eqref{eq:CD} and $u_\Dx$ be 
defined by \eqref{eq:PiecewiseConstantFuncDef}, where $u_\alpha$ is 
the solution to \eqref{eq:Semi-discreteScheme}.

Then, for any positive $R$ and $T$, there exists a 
constant $C$ depending only on $f$, $A$, $u_0$, $R$ and $T$, such that
\begin{equation*}
    \norm{u_\Dx(t)-u(t)}_{L^1(B(0,R))} \le C \Dx^{\tfrac{2}{19+d}},
    \qquad t \in [0,T].
\end{equation*}
\end{theorem}

\section{Kinetic formulations}\label{sec:kinetic}
In this section we supply certain kinetic formulations 
of the continuous and discrete equations \eqref{eq:CD} 
and \eqref{eq:Semi-discreteScheme}. 
As a preparation for the error estimate, we also regularize the kinetic 
equations by mollification. As explained in the 
introduction, due to the application of the viscous approximations in the 
proof of the error estimate, we 
assume $A' > 0$ for these intermediate results.

\subsection{Kinetic formulation of convection-diffusion equation}
\begin{lemma}\label{lemma:CarriloEntropyLemma}
  Assume that $A' > 0$ and set $B := A^{-1}$. Let $u$ be the solution
  of \eqref{eq:CD}. Define $g$ by $g(A(z)) = f(z)$ for all $z \in
  \R$. Let $S \in C^2(\R)$, 
  \begin{align*}
    \psi(u) &= \int_0^u S'(z)B'(z)\,dz,\qquad \psi_A(u) = \psi(A(u)),\\
    q(u) &= \int_0^u S'(z)g'(z)\,dz, \qquad q_A(u) = q(A(u)),
  \end{align*}
  and $S_A(u) = S(A(u))$. Then
  \begin{equation*}
    \partial_t\psi_A(u) + \nabla \cdot q_A(u) -\Delta S_A(u) =
    -S_A''(u)\abs{\nabla  A(u)}^2.
  \end{equation*}
\end{lemma}
\begin{proof}
  Multiplying \eqref{eq:CD} by
  $\psi_A'(u)$ gives
  \begin{equation*}
    \partial_t\psi_A(u) + \psi_A'(u)\nabla \cdot f(u) = \psi_A'(u)\Delta A(u).
  \end{equation*}
  Using a change of variables $A(\sigma) = z$,
  \begin{align*}
    q_A'(u) &= \partial_u\Bigl(\int_0^{A(u)} S'(z)g'(z)\,dz\Bigr) \\
    &= \partial_u\Bigr(\int_0^u
      S'(A(\sigma))g'(A(\sigma))A'(\sigma)\,d\sigma\Bigl) =
    S'(A(u))f'(u).
  \end{align*}
  Hence
  \begin{equation*}
    \psi_A'(u)\nabla \cdot f(u) = \nabla \cdot q_A(u).
  \end{equation*}
  Similarly we obtain $\psi_A'(u) = S'(A(u))$. Finally, observe that
  \begin{equation*}
    \Delta S_A(u) = S''(A(u))\abs{\nabla A(u)}^2 + \psi_A'(u)\Delta A(u).
  \end{equation*}
\end{proof}

The above entropy equation can be rephrased in terms of the $\chi$
function. Recall that for any locally Lipschitz 
continuous $\Psi:\R \rightarrow \R$,
\begin{equation}\label{eq:ReprOfComposedFuncChiFunc}
  \Psi(u)-\Psi(0) = \int_\R \Psi'(\xi)\chi(u;\xi)\,d\xi, \qquad (u \in \R).
\end{equation}
 
The next lemma reveals the equation
satisfied by $\chi(A(u);\zeta)$, where $u$ solves 
\eqref{eq:CD}, i.e., the kinetic formulation of 
the convection-diffusion equation.
 
\begin{lemma}\label{lemma:KineticEquationATransformed}
  Assume that $A' > 0$ and set $B:= A^{-1}$. 
  Let $u$ be the solution of \eqref{eq:CD}. Define
  $\rho(t,x,\zeta) = \chi(A(u(t,x));\zeta)$. Then
  \begin{equation}\label{eq:KineticEquationATransformed}
    \begin{cases}
      B'(\zeta)\partial_t \rho + g'(\zeta) \cdot \nabla \rho - \Delta
      \rho = \partial_\zeta m & \text{ in $\mathcal{D}'((0,T) \times
        \R^d \times \R)$}, 
      \\
      \rho(0,x,\zeta) = \chi(A(u_0(x));\zeta), & (x,\zeta) \in \R^d
      \times \R,
    \end{cases}
  \end{equation}
  where
  \begin{equation*}
    m(t,x,\zeta) = \delta(\zeta-A(u))\abs{\nabla A(u)}^2,
  \end{equation*}
  and $g$ satisfies $g(A(z)) = f(z)$ for all $z \in \R$.
\end{lemma}
\begin{proof}
  By Lemma~\ref{lemma:CarriloEntropyLemma} 
  and \eqref{eq:ReprOfComposedFuncChiFunc},
  \begin{multline*}
    \partial_t \int_\R S'(\zeta)B'(\zeta)\chi(A(u);\zeta)\,d\zeta +
    \nabla \cdot \int_\R S'(\zeta)g'(\zeta)\chi(A(u);\zeta)\,d\zeta 
    \\
    - \Delta \int_\R S'(\zeta)\chi(A(u);\zeta)\,d\zeta = \int_\R
    S'(\zeta) \partial_\zeta m(t,x,\zeta)\,d\zeta.
  \end{multline*}
\end{proof}

\subsection{Kinetic formulation of discrete equations}
\label{subsec:kinetic}
Stability/uniqueness analysis for differential equations 
often revolve around the chain rule.
The chain rule breaks down for numerical methods, but 
for us the next lemma will act as a substitute.
\begin{lemma}\label{lemma:NumericalChainRule}
  Let $S \in C^2(\R)$ satisfy $S'(0) = 0$. For any $g \in C^1(\R)$ and
  any real numbers $a$, $b$ and $c$,
  \begin{multline*}
    S'(a)(g(b)-g(a)) = \int_0^bS'(z)g'(z)\,dz - \int_0^a
    S'(z)g'(z)\,dz \\ + \int_a^bS''(z)(g(z)-g(b))\,dz.
  \end{multline*}
\end{lemma}
\begin{proof}
  For any $\zeta \in \R$, integration by parts yields
  \begin{equation*}
    S'(\zeta)(g(\zeta)-g(b)) = \int_0^\zeta S'(z)g'(z)\,dz +
    \int_0^\zeta S''(z)(g(z)-g(b))\,dz. 
  \end{equation*}
  Take the two equations obtained by setting $\zeta$ be equal to $a$
  and $b$ and subtract one from the other.
\end{proof}

To make the discrete and continuous calculus notations similar, 
we introduce the discrete gradient 
\begin{equation*}
  D_\pm\sigma = (D^1_\pm\sigma, \dots, D^d_\pm\sigma), \quad
   \mbox{ for any $\sigma:\Z^d \rightarrow \R$.}
\end{equation*}

The upcoming lemma contains the equation
satisfied by $\chi(u_\alpha;\zeta)$, where $u_\alpha$ 
is the solution of the scheme \eqref{eq:Semi-discreteScheme}.
\begin{lemma}\label{lemma:KineticFormOfSemidiscreteScheme}
  Suppose $ A' \ge 0$. Let $\seq{u_\alpha}_{\alpha \in \Z^d}$ be the
  solution to \eqref{eq:Semi-discreteScheme}. 
  Then $\rho_\alpha(t,\xi):= \chi(u_\alpha(t);\xi)$ satisfies
  \begin{align*}
      & \partial_t \rho + \left(F_1'(\xi) \cdot \Dm + F_2'(\xi) \cdot
        \Dp \right)\rho - A'(\xi)\Dm \cdot \Dp \rho = \partial_\xi(m_F
      + m_A), \\
      & \zeta_\alpha(0,\xi) = \chi(u_{\alpha,0};\xi),
  \end{align*}
  in $\mathcal{D}'(\R \times [0,T])$ for each $\alpha \in \Z^d$, where
  \begin{align*}
    m_F  = \sum_{i = 1}^d
    \Biggl((F_1^i(\xi) & -F_1^i(u_{\alpha-e_i}))\Dm^i\chi(u_\alpha;\xi)
    \\ & \quad  +
      (F_2^i(\xi)-F_2^i(u_{\alpha+e_i}))\Dp^i\chi(u_\alpha;\xi)\Biggr) 
  \end{align*}
  and
  \begin{align*}
    m_A = \sum_{i = 1}^d
    \Biggl(\frac{1}{\Dx}(A(u_{\alpha+e_i}) & -A(\xi))\Dp^i\chi(u_\alpha;\xi)
      \\ & \quad + \frac{1}{\Dx}(A(\xi)-A(u_{\alpha-e_i}))\Dm^i
      \chi(u_\alpha;\xi)\Biggr). 
  \end{align*}
\end{lemma}

\begin{proof}
  Since $\seq{u_\alpha}$ is a solution of
  \eqref{eq:Semi-discreteScheme},
  \begin{multline}\label{eq:SemiDiscTimesEnt}
    S'(u_\alpha(t))\partial_tu_\alpha(t) + \sum_{i = 1}^d
    S'(u_\alpha(t))\Dm^i F^i(u_\alpha(t),u_{\alpha+e_i}(t)) \\
    = \sum_{i = 1}^d S'(u_\alpha(t))\Dm^i\Dp^i A(u_\alpha(t)),
  \end{multline}
  for all $t \in (0,T)$ and $\alpha \in \Z^d$. By the chain rule
  \begin{equation*}
    S'(u_\alpha(t))\partial_tu_\alpha(t) = \partial_tS(u_\alpha(t)).
  \end{equation*}
  Consider the flux term. For each $i$, we have that
  $F^i=F_2^i+F_2^i$, and therefore
  \begin{equation*}
    S'(u_\alpha(t))\Dm^i F^i(u_\alpha(t),u_{\alpha+e_i}(t)) =
    S'(u_\alpha(t))\Dm^i F_1^i(u_\alpha) + S'(u_\alpha(t))\Dp^i
    F_2^i(u_\alpha). 
  \end{equation*}
  By Lemma~\ref{lemma:NumericalChainRule}, with $g$ equal to $F_1^i$ and
  $F_2^i$, we obtain
  \begin{align*}
    S'(u_\alpha(t))\Dm^i F_1^i(u_\alpha) &= \Dm^i Q_1^i(u_\alpha) -
    \frac{1}{\Dx}\int_{u_\alpha}^{u_{\alpha-e_i}}S''(z)(F_1^i(z)-F_1^i(u_{\alpha-e_i}))\,dz,
    \\ 
    S'(u_\alpha(t))\Dp^i F_2^i(u_{j}) &= \Dp^i Q_2^i(u_\alpha) +
    \frac{1}{\Dx}\int_{u_\alpha}^{u_{\alpha +
        e_i}}S''(z)(F_2^i(z)-F_2^i(u_{\alpha + e_i}))\,dz,
  \end{align*}
  where
  \begin{equation*}
    Q_j^i(u) := \int_0^u S'(z)(F_j^i)'(z)\,dz \,\mbox{ for $j = 1,2$.}
  \end{equation*}
  Consider the term on the right-hand side of \eqref{eq:SemiDiscTimesEnt}. Let
  \begin{equation*}
    R(u) := \int_0^u S'(z)A'(z)\,dz.
  \end{equation*}
  Fix $i$ and apply Lemma~\ref{lemma:NumericalChainRule} with $g = A$,
  $a = u_\alpha$, $b = u_{\alpha-e_i}$, and $u_{\alpha+e_i}$. Adding
  the two equations yields
  \begin{align*}
    S'(u_\alpha)\Dm^i \Dp^i(A(u_\alpha)) &= \Dm^i \Dp^i R(u_\alpha) \\
    &\quad
    +\frac{1}{\Dx^2}\int_{u_\alpha}^{u_{\alpha+e_i}}S''(z)(A(z)-A(u_{\alpha+e_i}))\,dz
    \\
    &\quad +
    \frac{1}{\Dx^2}\int_{u_\alpha}^{u_{\alpha-e_i}}S''(z)(A(z)-A(u_{\alpha-e_i}))\,dz.
  \end{align*}
  Hence \eqref{eq:SemiDiscTimesEnt} turns into 
  \begin{align*}
    \partial_tS(u_\alpha) + \sum_{i=1}^d\bigl(\Dm^i Q_1^i(u_\alpha) &+
    \Dp^i Q_2^i(u_\alpha)\bigr) -\sum_{i=1}^d\Dm^i\Dp^i
    R(u_\alpha)  \\
    &= \sum_{i=1}^d\frac{1}{\Dx}\int_{u_\alpha}^{u_{\alpha-e_i}}
    S''(z)(F_1^i(z)-F_1^i(u_{\alpha-e_i}))\,dz
    \\
    &\qquad -\sum_{i=1}^d\frac{1}{\Dx}\int_{u_\alpha}^{u_{\alpha+e_i}}
    S''(z)(F_2^i(z)-F_2^i(u_{\alpha+e_i}))\,dz
    \\
    &\qquad
    +\sum_{i=1}^d\frac{1}{\Dx^2}\int_{u_\alpha}^{u_{\alpha+e_i}}
    S''(z)(A(z)-A(u_{\alpha+e_i}))\,dz  \\
    &\qquad
    +\sum_{i=1}^d\frac{1}{\Dx^2}\int_{u_\alpha}^{u_{\alpha-e_i}}
    S''(z)(A(z)-A(u_{\alpha-e_i}))\,dz.
  \end{align*}
  By equation~\eqref{eq:ReprOfComposedFuncChiFunc},
  \begin{equation*}
    \Dm^i Q_1^i(u_\alpha) + \Dp^i Q_2^i(u_\alpha) = \int_\R
    S'(\xi)((F_1^i)'(\xi)\Dm^i + (F_2^i)'(\xi)\Dp^i)
    \chi(u_\alpha;\xi)\,d\xi. 
  \end{equation*}
  Similarly,
  \begin{equation*}
    \Dm^i\Dp^i R(u_\alpha) = \int_\R S'(\xi)A'(\xi) \Dm^i\Dp^i
    \chi(u_\alpha;\xi) \,d\xi. 
  \end{equation*}
  Consider the right-hand side. For any $g \in C(\R)$,
  \begin{equation*}
    \int_a^b S''(z)(g(z)-g(b)) \,dz = \int_\R
    S''(\xi)(g(\xi)-g(b))\left(\chi(b;\xi)-\chi(a;\xi)\right)\,d\xi. 
  \end{equation*}
  Hence
  \begin{align*}
    &\frac{1}{\Dx}
    \int_{u_\alpha}^{u_{\alpha-e_i}}(F_1^i(z)-F_1^i(u_{\alpha-e_i})\,dz
    = -\int_\R S''(\xi)(F_1^i(\xi)-F_1^i(u_{\alpha-e_i}))\Dm^i
    \chi(u_\alpha;\xi)\,d\xi, \\
    &\frac{-1}{\Dx}
    \int_{u_\alpha}^{u_{\alpha+e_i}}(F_2^i(z)-F_2^i(u_{\alpha+e_i})\,dz
    = -\int_\R S''(\xi)(F_2^i(\xi)-F_2^i(u_{\alpha+e_i}))\Dp^i
    \chi(u_\alpha;\xi)\,d\xi.
  \end{align*}
  Similarly,
  \begin{align*}
    \frac{1}{\Dx^2}\int_{u_\alpha}^{u_{\alpha+e_i}}S''(z)&(A(z)-A(u_{\alpha+e_i}))\,dz
    \\
    &=\frac{-1}{\Dx}\int_\R S''(\xi)
    (A(u_{\alpha+e_i})-A(\xi))\Dp^i\chi(u_\alpha;\xi)\,d\xi, \\
    \frac{1}{\Dx^2}\int_{u_\alpha}^{u_{\alpha-e_i}}S''(z)&(A(z)-A(u_{\alpha-e_i}))\,dz
    \\
    &=\frac{-1}{\Dx}\int_\R S''(\xi)
    (A(\xi)-A(u_{\alpha-e_i}))\Dm^i\chi(u_\alpha;\xi)\,d\xi.
  \end{align*}
  The result follows.
\end{proof}
For a function $u: \R^d \rightarrow \R$ we define the shift 
operator $S_y$ by $S_yu(x) = u(x+y)$. Then the discrete derivatives
may be expressed as
\begin{equation*}
  D^i_\pm u = \pm\frac{S_{\pm \Dx_i}u-u}{\Dx},
\end{equation*}
where $\Dx_i = \Dx\, e_i$. 

Making a change of variable $\zeta= A(\xi)$, we can obtain 
an equation satisfied by $\chi(A(u_\Dx);\zeta)$, where $u_\Dx$ is the 
numerical solution \eqref{eq:PiecewiseConstantFuncDef},  
resulting in the ``discrete" kinetic formulation to be utilized later.
\begin{lemma}\label{lemma:KineticSemiDiscEquATransformed}
  Suppose $A' > 0$. Let $\seq{u_\alpha}$ be the solution to
  \eqref{eq:Semi-discreteScheme} and define  $u_\Dx$ by 
  \eqref{eq:PiecewiseConstantFuncDef}. Let $G_j:\R \rightarrow \R^d$
  satisfy $G_j(A(u)) = F_j(u)$ $\forall u$, for $j = 1$, $2$. Then
  $\rho^\Dx(t,x,\zeta) = \chi(A(u_\Dx(t,x));\zeta)$ satisfies
    \begin{align*}
     & B'(\zeta)\partial_t \rho^\Dx + (G_1'(\zeta) \cdot \Dp +
      G_2'(\zeta) \cdot \Dm)\rho^\Dx 
      -\Dm \cdot \Dp \rho^\Dx
      = \partial_\zeta(n_A^\Dx + n_G^\Dx), 
      \\  & \rho^\Dx(0,\zeta) = \chi(A(u^0_\Dx);\zeta),
  \end{align*}
  in $\mathcal{D}'(\R \times \Pi_T)$, where
  \begin{multline*}
    n_G^\Dx = \sum_{i=1}^d
    (G_1^i(\zeta)-G_1^i(A(S_{-\Dx_i}u_\Dx)))\Dm^i \chi(A(u_\Dx);\zeta)
    \\ + \sum_{i=1}^d(G_2^i(\zeta)-G_2^i(A(S_{\Dx_i}u_\Dx)))\Dp^i
    \chi(A(u_\Dx);\zeta)
  \end{multline*}
  and
  \begin{multline}\label{eq:nadxdef}
    n_A^\Dx =
    \sum_{i=1}^d\frac{1}{\Dx}(A(S_{\Dx_i}u_{\Dx})-\zeta)\Dp^i\chi(A(u_\Dx);\zeta)
    \\ 
    + \sum_{i=1}^d\frac{1}{\Dx}(\zeta-A(S_{-\Dx_i}u_\Dx))\Dm^i
    \chi(A(u_\Dx);\zeta).
  \end{multline}
\end{lemma}
\begin{proof}
  Let $S \in C_c^\infty(\R)$ and define $S_A(\xi) = S(A(\xi))$. 
  By Lemma~\ref{lemma:KineticFormOfSemidiscreteScheme},
  \begin{align*}
    &\partial_t\int_\R S_A(\xi) \chi(u_\alpha;\xi)\,d\xi +
    \int_\R S_A(\xi)(F'_1(\xi) \cdot \Dm + F'_2(\xi) \cdot \Dp)\chi(u_\alpha;\xi)\,d\xi \\
    &- \int_\R S_A(\xi)A'(\xi)\Dm \cdot \Dp \chi(u_\alpha;\xi)\,d\xi =
    -\int_\R S_A'(\xi)(m_F + m_A)\,d\xi.
  \end{align*}
  Let $\zeta = A(\xi)$ and note that $\chi(u_\alpha;\xi) =
  \chi(A(u_\alpha);A(\xi))$. The terms on the left-hand side are
  straightforward to verify. Next,
  \begin{align*}
    \int_\R S_A'(\xi)&m_F(\xi)\,d\xi \\
    & = \sum_{i=1}^d \int_\R
    S'(A(\xi))(G_1^i(A(\xi))-G_1^i(A(u_{\alpha-e_i})))\Dm^i
    \chi(A(u_\alpha);A(\xi)) A'(\xi) \,d\xi \\
    & \quad + \sum_{i=1}^d\int_\R
    S'(A(\xi))(G_2^i(A(\xi))-G_2^i(A(u_{\alpha+e_i})))\Dp^i
    \chi(A(u_\alpha);A(\xi)) A'(\xi) \,d\xi \\
    &= \sum_{i=1}^d\int_\R S'(\zeta)
    (G_1^i(\zeta)-G_1^i(A(u_{\alpha-e_i})))\Dm^i
    \chi(A(u_\alpha);\zeta)\,d\zeta \\
    & \quad + \sum_{i=1}^d\int_\R
    S'(\zeta)(G_2^i(\zeta)-G_2^i(A(u_{\alpha+e_i})))\Dp^i
    \chi(A(u_\alpha);\zeta) \,d\zeta.
  \end{align*}
  A similar computation shows the 
  second equality involving $n_A^\Dx$.
\end{proof}

\subsection{Various regularizations}
In this section we study mollified versions of 
Lemma~\ref{lemma:KineticEquationATransformed} and
\ref{lemma:KineticSemiDiscEquATransformed}. Let us first introduce 
some notation. Let $J \in C^\infty_c(\R)$ denote a function satisfying
\begin{equation*}
  \mbox{supp}(J) \subset [-1,1], \quad \int_{\R} J(x) \,dx = 1 \mbox{ and } J(-x) = J(x)
\end{equation*}
for all $x \in \R$. That is, $J$ is a symmetric mollifier on $\R$ with
support in $[-1,1]$. For any $\sigma > 0$ we let $J_\sigma(x) =
\sigma^{-1}J(\sigma^{-1}x)$. For any $n \ge 1$, $J^{\otimes n}_\sigma$
is a symmetric mollifier on $\R^n$ with support in
$[-\sigma,\sigma]^n$. In general the dimension of the argument will
define $n$, so to simplify the notation we write $J_\sigma$ instead of
$J^{\otimes n}_\sigma$.

Let $\psi:\R^2 \rightarrow \R$ be a continuous function and 
$u, v\in L^1(\R)$. Then we define
\begin{equation*}
  (\psi(u,v) \conv{u,v} f \otimes g)(x) 
  := \int_\R\int_\R \psi(u(y_1),v(y_2))f(x-y_1)g(x-y_2)\,dy_1dy_2,
\end{equation*}
where $f,g \in L^1(\R)$. Similarly, we let
\begin{equation*}
  \psi(u,v) \conv{u} f := \int_\R \psi(u(y),v(x))f(x-y)\,dy. 
\end{equation*}
This notation generalizes in an obvious way 
to functions of several variables.

We start by introducing regularizations of 
$\sgn(\cdot)$ and $\chi(u;\cdot)$
\begin{lemma}\label{lemma:MollifiedChiFuncProp}
  For $\eps > 0$, define
  \begin{equation*}
    \signe{\xi} := 2\int_0^\xi J_\eps(\zeta)\,d\zeta, 
    \quad 
    \chi_\eps(u;\xi) := \int_\R \chi(u;\zeta)
    J_\eps(\xi-\zeta)\,d\zeta.
  \end{equation*}
  Then
  \begin{itemize}
  \item[(i)] For each $\xi$, $u \mapsto \chi_\eps(u;\xi)
    \in C^\infty(\R)$ and $\partial_u \chi_\eps(u;\xi) =
    J_\eps(\xi-u)$.
  \item[(ii)] For all $u$ and $\xi$
    \begin{equation*}
      \signe{\xi}-2\chi_\eps(u;\xi) = \signe{\xi-u}.
    \end{equation*}
  \item[(iii)] For any $u$
    \begin{equation*}
      \int_\R \abs{\chi_\eps(u;\xi)-\chi(u;\xi)} \,d\xi \le 4\eps.
    \end{equation*}
  \end{itemize}
\end{lemma}
\begin{proof}
  We first prove (i). Let $H_\eps'(\sigma) =
  J_\eps(\sigma)$. Since $J_\eps(\xi-\zeta) =
  J_\eps(\zeta-\xi)$,
  \begin{multline*}
    \lim_{h \rightarrow 0}\frac{1}{h}\int_\R (\chi(u+h;\zeta)-\chi(u;\zeta))
    J_\eps(\xi-\zeta)\,d \zeta \\
    = \lim_{h \rightarrow 0}\frac{1}{h}
    \left(H_\eps(u+h-\xi)-H_\eps(u-\xi)\right)
    = J_\eps(u-\xi).
  \end{multline*}
  Next we prove (ii). Let $\sigma = \zeta-\xi$. 
  By the symmetry of $J_\eps$,
  \begin{equation*}
    \chi_\eps(u;\xi) = \int_\R \chi(u;\sigma +
    \xi)J_\eps(\sigma)\,d\sigma.  
  \end{equation*}
  A calculation (or \eqref{eq:ChiDiffEqu}) yields
  \begin{equation*}
    \chi(u;\sigma + \xi) = \chi(u-\xi;\sigma)-\chi(-\xi;\sigma). 
  \end{equation*}
  Note that $\chi(-\xi;\sigma) = -\chi(\xi;-\sigma)$. Hence
  \begin{equation*}
    \chi_\eps(u;\xi) = \int_\R
    (\chi(u-\xi;\zeta)+\chi(\xi;-\zeta))J_\eps(\zeta)\,d\zeta. 
  \end{equation*}
  It follows that
  \begin{align*}
    \signe{\xi}-2\chi_\eps(u;\xi) &= -2 \int_\R
    \chi(u-\xi;\zeta)J_\eps(\zeta)\,d\zeta
    + 2\int_\R (\chi(\xi;\zeta)-\chi(\xi;-\zeta))J_\eps(\zeta)\,d\zeta \\
    &=: \term_1 + \term_2.
  \end{align*}
  Since $(\chi(\xi;\zeta)-\chi(\xi;-\zeta))$ is antisymmetric in
  $\zeta$ and $J_\eps$ is symmetric it follows that $\term_2 =
  0$. Now
  \begin{equation*}
    \term_1 = -2 \int_0^{u-\xi}J_\eps(\zeta)\,d\zeta = 2
    \int_0^{\xi-u}J_\eps(\zeta)\,d\zeta = \signe{\xi-u}. 
  \end{equation*}
  To prove (iii), note that
  \begin{equation*}
    \abs{\chi_\eps(u;\xi)-\chi(u;\xi)} = 
    0 \text{ whenever } \xi \notin (-\eps,\eps) \cup
    (u-\eps,u + \eps). 
  \end{equation*}
\end{proof}
For $\eps > 0$ and  $f \in C(\R)$, let
$R^f_\eps:\R^2 \rightarrow \R$ be defined by
\begin{equation}\label{eq:RfDefinition}
  \int_\R f(\sigma)\chi(u;\sigma)J_\eps(\zeta-\sigma)\,d\sigma
  = R^f_\eps(u,\zeta) + f(\zeta)\chi_\eps(u;\zeta),
\end{equation}
for all $u,\zeta \in \R$. 

Now we are ready to provide ``regularized" versions
of Lemmas \ref{lemma:KineticEquationATransformed} and
\ref{lemma:KineticSemiDiscEquATransformed}. As the mollification 
will take place on a slightly smaller region, we introduce the notation 
$\Pi_T^{r_0} := (r_0, T-r_0) \times \R^d$. 

We start with the regularization 
of the kinetic formulation of the convection-diffusion equation.  
\begin{lemma}\label{lemma:KineticFormATransformedMollified}
  Assume that $A' > 0$. Let $u$ be the solution of
  \eqref{eq:CD} and define
  \begin{equation*}
    \rho_{\eps,r,r_0} := 
    \chi(A(u);\cdot) \star J_{r_0} 
    \otimes J_r \otimes J_\eps. 
  \end{equation*}
  Then for $(t,x,\zeta)\in \Pi_T^{r_0} \times \R$, the function 
  $\rho_{\eps,r,r_0}$ satisfies
  \begin{equation*}
    B'(\zeta)\partial_t \rho_{\eps,r,r_0} + g'(\zeta) \cdot
    \nabla \rho_{\eps,r,r_0} - \Delta \rho_{\eps,r,r_0}
    + \partial_t R_{\eps,r,r_0}^{B'} + \nabla \cdot
    R_{\eps,r,r_0}^{g'} = \partial_\zeta
    n_{A,\eps,r,r_0},
  \end{equation*}
  where
  \begin{equation*}
    R^f_{\eps,r,r_0} = R^f_\eps(A(u),\zeta) \star J_r
    \otimes J_{r_0},
  \end{equation*}
  with $R^f_\eps$ defined by
  \eqref{eq:RfDefinition}, and
  \begin{equation*}
    n_{A,\eps,r,r_0}(t,x,\zeta) =
    \left(J_\eps(\zeta-A(u)) \abs{\nabla A(u)}^2 \star J_{r_0}
      \otimes J_r\right)(t,x). 
  \end{equation*}
\end{lemma}
\begin{proof}
Starting off from Lemma~\ref{lemma:KineticEquationATransformed}, take the
convolution of equation \eqref{eq:KineticEquationATransformed} with
$J_\eps$ and apply \eqref{eq:RfDefinition}. 
Finally, convolve the resulting equation 
with $J_r \otimes J_{r_0}$.
\end{proof}

Next up is the regularization of the kinetic formulation of the discrete 
equations.

\begin{lemma}\label{lemma:KineticFormulationOfSemiDiscATransMollified}
  Under the same assumptions and with the same notation as in
  Lemma~\ref{lemma:KineticSemiDiscEquATransformed}, define
  \begin{equation*}
    \rho_{\eps,r,r_0}^{\Dx} :=  \chi(A(u_\Dx);\cdot) \star
    J_{r_0} \otimes J_r \otimes J_\eps.  
  \end{equation*}
  For $(t,x,\zeta)\in \Pi_T^{r_0} \times \R$, the function 
  $\rho_{\eps,r,r_0}^\Dx$ satisfies
  \begin{multline*}
    B'(\zeta)\partial_t \rho_{\eps,r,r_0}^\Dx + g'(\zeta) \cdot
    \nabla \rho_{\eps,r,r_0}^\Dx
    -\Delta\rho_{\eps,r,r_0}^\Dx + G_1'(\zeta) \cdot (\Dp -
    \nabla)\rho_{\eps,r,r_0}^\Dx \\ 
    +G_2'(\zeta) \cdot(\Dm - \nabla)\rho_{\eps,r,r_0}^\Dx +
    (\Delta-\Dm \cdot \Dp) \rho_{\eps,r,r_0}^\Dx + \partial_t
    R^{B',\Dx}_{\eps,r,r_0} \\ 
    + \Dp \cdot R^{G_1',\Dx}_{\eps,r,r_0} + \Dm \cdot
    R^{G_2',\Dx}_{\eps,r,r_0}
    = \partial_\zeta(n^\Dx_{A,\eps,r,r_0} +
    n^\Dx_{G,\eps,r,r_0}).
  \end{multline*}
  Here, $R^{f,\Dx}_{\eps,r,r_0} =
  R_\eps^f(A(u_\Dx),\cdot) \star J_r \otimes J_{r_0}$ with
  $R^f_\eps$ coming from \eqref{eq:RfDefinition}. Furthermore,
  \begin{equation*}
    n^\Dx_{A,\eps,r,r_0} = n^\Dx_A \star (J_\eps \otimes J_r \otimes J_{r_0})
    \mbox{ and }
    n^\Dx_{G,\eps,r,r_0} = n^\Dx_G \star (J_\eps \otimes J_r \otimes J_{r_0}).
  \end{equation*}
\end{lemma}
\begin{proof}
 In view of Lemma~\ref{lemma:KineticSemiDiscEquATransformed} and 
 \eqref{eq:RfDefinition},
  \begin{multline*}
    B'(\zeta)\partial_t \rho_\eps^\Dx + (G_1'(\zeta) \cdot \Dp 
    + G_2'(\zeta) \cdot \Dm)\rho_\eps^\Dx -\Dm \cdot \Dp \rho_\eps^\Dx \\
    + \partial_t R^{B'}_\eps(u_\Dx,\zeta) + \Dp \cdot
    R^{G_1'}_\eps(u_\Dx,\zeta) + \Dm \cdot
    R^{G_2'}_\eps(u_\Dx,\zeta)
    = \partial_\zeta(n^\Dx_{A,\eps} + n^\Dx_{G,\eps}),
  \end{multline*}
  where $\rho^\Dx_\eps(t,x,\zeta) =
  \chi_\eps(A(u_\Dx);\zeta)$ and $n^\Dx_{A,\eps} =
  n^\Dx_A \star J_\eps$ and $n^\Dx_{G,\eps} = n^\Dx_G
  \star J_\eps$. Take the convolution of the above equation with
  $J_r \otimes J_{r_0}$. Recall that $G_1' + G_2' = g'$ and add and
  subtract to obtain the result.
\end{proof}

\section{Proof of Theorem \ref{theorem:MultidLocal}}
\label{sec:errest}
We are now ready to embark on the proof of the error 
estimate (Theorem \ref{theorem:MultidLocal}). 
Instead of working directly with the microscopic contraction
functional \eqref{eq:QDef}, we introduce a 
regularized version $Q_\varepsilon$ of it.
For $u,v, \xi \in \R$, define
\begin{equation}\label{eq:Lepsdef}
  Q_{\eps}(u,v;\xi) := \signe{\xi}\chi_\eps(u;\xi) +
  \signe{\xi}\chi_\eps(v;\xi)-2\chi_\eps(u;\xi)\chi_\eps(v;\xi), 
\end{equation}
where $\mathrm{sign}_\eps$ and $\chi_\eps$ are given
in Lemma~\ref{lemma:MollifiedChiFuncProp}. One may show that
\begin{displaymath}
 \int_\R (\chi_\varepsilon(u;\xi)-\chi_\varepsilon(v;\xi))^2\,d\xi 
 = \int_\R Q_\varepsilon(u,v;\xi)\,d\xi.
\end{displaymath}
This equality is, however, not directly useful to us, since we 
will be working with functions like $\chi_\varepsilon(A(u);\xi)$ 
with $A(\cdot)$ nonlinear, but see the related Lemma~\ref{lem:absdiffs}.

\subsection{Main error equation}
We will use the kinetic formulations of the convection-diffusion equation 
and the difference method to derive a fundamental equation for the error quantity 
$Q_\eps(A(u(t,x)),A(u_\Dx(t,x));\zeta)$ (properly regularized).

\begin{lemma}\label{lemma:ContractionLemma}
  Assume that $A' > 0$. With the notation of
  Lemma~\ref{lemma:KineticFormATransformedMollified} and
  \ref{lemma:KineticFormulationOfSemiDiscATransMollified}, define
  \begin{equation*}
    Q_{\eps,r,r_0}(\zeta) = Q_\eps(A(u),A(u_\Dx);\zeta) \conv{u,u_\Dx}
    (J_{r_0} \otimes J_r) \otimes (J_{r_0} \otimes J_r). 
  \end{equation*}
  Then, for all $(t,x) \in \Pi_T^{r_0}$,
  \begin{align}
    \int_\R &B'(\zeta) \partial_tQ_{\eps,r,r_0}\,d\zeta +
    \int_\R g'(\zeta) \cdot \nabla Q_{\eps,r,r_0}\,d\zeta
    \label{eq:T0}
    \\&= \int_\R \Delta Q_{\eps,r,r_0}\,d\zeta 
    +2\int_\R \nabla \rho_{\eps,r,r_0} \cdot (2\nabla
    -(\Dp + \Dm))\rho_{\eps,r,r_0}^\Dx\,d\zeta \label{eq:T1}
    \\ 
    &\quad -\int_\R (\signe{\zeta-A(u_\Dx)} \star J_{r_0} \otimes
    J_r)\partial_tR_{\eps,r,r_0}^{B'}\,d\zeta\label{eq:T2}
    \\ 
    &\quad -\int_\R (\signe{\zeta-A(u)} \star J_{r_0} \otimes
    J_r)\partial_t R^{B',\Dx}_{\eps,r,r_0}\,d\zeta \label{eq:T3}
    \\
    &\quad -\int_\R (\signe{\zeta-A(u_\Dx)} \star J_{r_0} \otimes
    J_r)\nabla \cdot R_{\eps,r,r_0}^{g'} \,d\zeta\label{eq:T4}
    \\
    &\quad -\int_\R (\signe{\zeta-A(u)} \star J_{r_0} \otimes
    J_r)\left(\Dp \cdot R_{\eps,r,r_0}^{G_1',\Dx} + \Dm \cdot
      R_{\eps,r,r_0}^{G_2',\Dx}\right)  \, d\zeta\label{eq:T5}
    \\
    &\quad -\int_\R (\signe{\zeta-A(u)}\star J_{r_0} \otimes
    J_r)\big(G_1'(\zeta) \cdot (\Dp-\nabla) \notag
    \\
    &\hphantom{XXXXXXXXXXXXX} + G_2'(\zeta)\cdot
    (\Dm-\nabla)\big)\rho_{\eps,r,r_0}^\Dx\,d\zeta.\label{eq:T6}
    \\
    &\quad +\int_\R (\signe{\zeta-A(u)} \star J_{r_0} \otimes
    J_r)(\Delta-\Dm \cdot \Dp)\rho^\Dx_{\eps,r,r_0}\,d\zeta 
    \label{eq:T7}
    \\
    &\quad -2\int_\R (J_\eps(\zeta-A(u)) \star J_{r_0} \otimes
    J_r)n^\Dx_{G,\eps,r}\,d\zeta\label{eq:T8}
     \\
     &\quad
    -2\int_\R E_{\Dx,\eps,r,r_0}(\zeta)\,d\zeta,\label{eq:T9}
  \end{align}
  where
  \begin{equation}\label{eq:discretediss}
	  \begin{split}
    		E_{\Dx,\eps,r,r_0}(\zeta) &  = -\nabla\rho_{\eps,r,r_0}
    		\cdot (\Dp + \Dm)\rho_{\eps,r,r_0}^\Dx
		 \\ & \qquad\quad
		 +(J_\eps(\zeta-A(u_\Dx)) \star J_{r_0} \otimes J_r) n_{A,\eps,r,r_0}
		\\ & \qquad \quad
		+(J_\eps(\zeta-A(u))\star J_{r_0} 
		\otimes J_r)n^\Dx_{A,\eps,r,r_0}.
	  \end{split}
  \end{equation}

\end{lemma}
\begin{proof}
  By definition,
  \begin{multline*}
    Q_{\eps,r,r_0}(t,x,\zeta) =
    \signe{\zeta}\rho_{\eps,r,r_0}(t,x,\zeta)
    + \signe{\zeta}\rho^\Dx_{\eps,r,r_0}(t,x,\zeta) \\
    -2\rho_{\eps,r,r_0}(t,x,\zeta)\rho^\Dx_{\eps,r,r_0}(t,x,\zeta).
  \end{multline*}
  Hence,
  \begin{align*}
    \partial_t \int_\R Q_{\eps,r,r_0}B'(\zeta)\,d\zeta
    &= \int_\R \signe{\zeta}\partial_t(\rho_{\eps,r,r_0}
    +\rho^\Dx_{\eps, r,r_0})B'(\zeta)\,d\zeta \\
    &\quad + \int_\R \partial_t(\rho_{\eps,r,r_0}
    \rho^\Dx_{\eps,r,r_0})B'(\zeta)\,d\zeta \\
    &=:\term_1 + \term_2.
  \end{align*}
  By Lemmas \ref{lemma:KineticFormATransformedMollified} and
  \ref{lemma:KineticFormulationOfSemiDiscATransMollified}
  \begin{align*}
    \term_1 &= \underbrace{-\int_\R \signe{\zeta}g'(\zeta) \cdot \nabla
      (\rho_{\eps,r,r_0} +
      \rho^\Dx_{\eps,r,r_0})\,d\zeta}_{\term_1^1}
    \\
    &\quad + \underbrace{\int_\R
      \signe{\zeta}\Delta(\rho_{\eps,r,r_0}
      +\rho^\Dx_{\eps,r,r_0})\,d\zeta}_{\term_1^2}
    \\
    &\quad \underbrace{- \int_\R
      \signe{\zeta}\partial_t(R_{\eps,r,r_0}^{B'} +
      R_{\eps,r,r_0}^{B',\Dx})\,d\zeta}_{\term_1^3}
    \\
    &\quad \underbrace{- \int_\R \signe{\zeta}\left(\nabla \cdot
        R_{\eps,r,r_0}^{g'} + \Dp \cdot
        R_{\eps,r,r_0}^{G_1',\Dx} + \Dm \cdot
        R_{\eps,r,r_0}^{G_2',\Dx}\right)\,d\zeta}_{\term_1^4} \\
    &\quad + \underbrace{\int_\R \signe{\zeta}\left(\partial_\zeta
        n_{A,\eps,r,r_0} + \partial_\zeta
        n^\Dx_{A,\eps,r,r_0}\right)\,d\zeta}_{\term_1^5}
    \\
    &\quad \underbrace{- \int_\R \signe{\zeta}\left(G_1'(\zeta) \cdot
        (\Dp-\nabla) + G_2'(\zeta)\cdot
        (\Dm-\nabla)\right)\rho_{\eps,r,r_0}^\Dx\,d\zeta}_{\term_1^6}
    \\
    &\quad + \underbrace{\int_\R \signe{\zeta}(\Delta-\Dm \cdot
      \Dp)\rho^\Dx_{\eps,r,r_0}\,d\zeta}_{\term_1^7}
    \\
    &\quad + \underbrace{\int_\R \signe{\zeta}\partial_\zeta
      n^\Dx_{G,\eps,r,r_0}\,d\zeta}_{\term_1^8}.
  \end{align*}
  Similarly for $\term_2$ we obtain
  \begin{align*}
    \term_2 &= \underbrace{2\int_\R g'(\zeta) \cdot
      \nabla(\rho_{\eps,r,r_0}\rho^\Dx_{\eps,r,r_0})\,d\zeta}_{\term_2^1}
    \\
    &\quad \underbrace{- 2\int_\R \rho_{\eps,r,r_0}\Delta
      \rho^\Dx_{\eps,r,r_0} + \Delta
      \rho_{\eps,r,r_0}\rho^\Dx_{\eps,r,r_0} \,d\zeta}_{\term_2^2}
    \\
    &\quad + \underbrace{2\int_\R \rho_{\eps,r,r_0}\partial_t
      R^{B',\Dx}_{\eps,r,r_0} + \rho^\Dx_{\eps,r,r_0}\partial_t
      R^{B'}_{\eps,r,r_0} \,d\zeta }_{\term_2^3}
    \\
    &\quad + \underbrace{2\int_\R \rho^\Dx_{\eps,r,r_0}\nabla \cdot
      R_{\eps,r,r_0}^{g'} + \rho_{\eps,r,r_0}\Dp \cdot
      R_{\eps,r,r_0}^{G_1',\Dx} + \rho_{\eps,r,r_0}\Dm \cdot
      R_{\eps,r,r_0}^{G_2',\Dx} \, d\zeta}_{\term_2^4}
    \\
    &\quad \underbrace{- 2\int_\R \rho_{\eps,r,r_0} \partial_\zeta
      n^\Dx_{A,\eps,r,r_0} + \rho_{\eps,r,r_0}^\Dx \partial_\zeta
      n_{A,\eps,r,r_0} \,d\zeta}_{\term_2^5}
    \\
    &\quad + \underbrace{2\int_\R \rho_{\eps,r,r_0}\left(G_1'(\zeta)
        \cdot (\Dp - \nabla) + G_2'(\zeta) \cdot (\Dm
        -\nabla)\right)\rho_{\eps,r,r_0}^\Dx \, d\zeta}_{\term_2^6}
    \\
    &\quad \underbrace{- 2\int_\R \rho_{\eps,r,r_0}(\Delta - \Dm \cdot
      \Dp)\rho_{\eps,r,r_0}^\Dx\,d\zeta}_{\term_2^7}
    \\
    &\quad \underbrace{- 2\int_\R \rho_{\eps,r,r_0}\partial_\zeta
      n^\Dx_{G,\eps,r,r_0}\,d\zeta}_{\term_2^8}.
  \end{align*}
  We compute $\term_1 + \term_2$ term by term, and thereby explain each of the
  terms \eqref{eq:T1}--\eqref{eq:T9} in the lemma. We start with
  \begin{equation*}
    \term_1^1 + \term_2^1 = - \int_\R g'(\zeta) \cdot \nabla Q_{\eps,r,r_0}\,d\zeta,
  \end{equation*}
  which gives the last term in \eqref{eq:T0}.

  To make the second derivative terms a complete derivative we need to
  add and subtract. Hence we may write
  \begin{align*}
    \term_1^2 + \term_2^2 &= \Delta \int_\R Q_{\eps,r,r_0} \,d\zeta +
    4\int_\R \nabla \rho_{\eps,r,r_0} \cdot \nabla
    \rho_{\eps,r,r_0}^\Dx\,d\zeta \\
    &= \Delta \int_\R Q_{\eps,r,r_0} \,d\zeta + 2\int_\R \nabla
    \rho_{\eps,r,r_0} \cdot (\Dp +
    \Dm)\rho_{\eps,r,r_0}^\Dx\,d\zeta \\
    &\qquad +2\int_\R \nabla \rho_{\eps,r,r_0}
    (2\nabla -(\Dp +\Dm))\rho_{\eps,r,r_0}^\Dx\,d\zeta,
  \end{align*}
  which explains \eqref{eq:T1} and the first term in $E_{\Dx,\eps,r,r_0}$.

  By Lemma~\ref{lemma:MollifiedChiFuncProp} it follows that
  \begin{align*}
    \signe{\zeta}-2\rho_{\eps,r,r_0} &= \signe{\zeta-A(u)}
    \star J_{r_0} \otimes J_r, \\
    \signe{\zeta}-2\rho^\Dx_{\eps,r,r_0} &=
    \signe{\zeta-A(u_\Dx)} \star J_{r_0} \otimes J_r.
  \end{align*}
  Hence,
  \begin{align*}
    \term_1^3 + \term_2^3 &= - \int_\R (\signe{\zeta-A(u_\Dx)} \star J_{r_0}
    \otimes
    J_r)\partial_tR_{\eps,r,r_0}^{B'}\,d\zeta \\
    & \qquad -\int_\R (\signe{\zeta-A(u)} \star J_{r_0} \otimes
    J_r)\partial_tR^{B',\Dx}_{\eps,r,r_0}\,d\zeta,
  \end{align*}
  which explains \eqref{eq:T2} and \eqref{eq:T3}
  Similarly,
  \begin{align*}
    \term_1^4 + \term_2^4
    &= - \int_\R (\signe{\zeta-A(u_\Dx)} \star J_{r_0} \otimes
    J_r)\nabla \cdot R_{\eps,r,r_0}^{g'}\,d\zeta \\ 
    &\quad - \int_\R (\signe{\zeta-A(u)} \star J_{r_0} \otimes
    J_r)\left(\Dp \cdot R_{\eps,r,r_0}^{G_1',\Dx} + \Dm \cdot
      R_{\eps,r,r_0}^{G_2',\Dx}\right) \, d\zeta,
  \end{align*}
  which explains the presence 
  of \eqref{eq:T4} and \eqref{eq:T5}.

  Performing integration by parts we obtain, using
  Lemma~\ref{lemma:MollifiedChiFuncProp},
  \begin{align*}
    \term_1^5 + \term_2^5
    &=\int_\R (\signe{\zeta-A(u_\Dx)} \star J_{r_0} \otimes
    J_r)\partial_\zeta n_{A,\eps,r,r_0} 
    \\ 
    & \hphantom{=-2\int_\R}\quad+ (\signe{\zeta-A(u)}\star J_{r_0} \otimes
    J_r)\partial_\zeta n^\Dx_{A,\eps,r,r_0} \,d\zeta
    \\ 
    &=-2\int_\R (J_\eps(\zeta-A(u_\Dx)) \star J_{r_0} \otimes J_r)
    n_{A,\eps,r,r_0} \\
    &\hphantom{=-2\int_\R}\quad + (J_\eps(\zeta-A(u))\star J_{r_0} \otimes
    J_r)n^\Dx_{A,\eps,r,r_0} \,d\zeta, 
  \end{align*}
  which explains the two last terms in $E_{\Dx,\eps,r,r_0}$.

  Similarly,
  \begin{align*}
    \term_1^6 + \term_2^6
    &= - \int_\R (\signe{\zeta-A(u)}\star J_{r_0} \otimes
    J_r)\big(G_1'(\zeta) \cdot (\Dp-\nabla) 
    \\
    & \hphantom{= - \int_\R}\quad + G_2'(\zeta)\cdot
    (\Dm-\nabla)\big)\rho_{\eps,r,r_0}^\Dx\,d\zeta,
    \intertext{and}
    \term_1^7 + \term_2^7 &= \int_\R (\signe{\zeta-A(u)} \star J_{r_0} \otimes
    J_r)(\Delta-\Dm \cdot \Dp)\rho^\Dx_{\eps,r,r_0}\,d\zeta,
  \end{align*}
   explaining the terms \eqref{eq:T6} and \eqref{eq:T7}.

  Finally, integration by parts yields
  \begin{equation*}
    \term_1^8 + \term_2^8 = -2\int_\R (J_\eps(\zeta-A(u)) \star J_{r_0} \otimes
    J_r)n^\Dx_{G,\eps,r,r_0}\,d\zeta,  
  \end{equation*}
  which is the term \eqref{eq:T8}.
\end{proof}

\subsection{Dissipative term}
In this subsection we are 
concerned with finding an upper bound on 
\eqref{eq:discretediss}. In the continuous setting, this ``dissipative" term is 
negative, cf.~\eqref{eq:SignOfDissipTerm}, which comes 
as a consequence of the chain rule of calculus. 
The following elementary lemma will help us contend with the lack of 
a discrete chain rule.

\begin{lemma}\label{lemma:DiscreteChainRule2Order}
  Let $a$ and $b$ be two real numbers. 
  Then there exist real numbers 
  $\tau = \tau_\eps(a,b,\zeta)$ and $\theta =
  \theta_\eps(a,b,\zeta)$ such that $\tau$ and $\theta$ 
  are between $a$ and $b$,  and
  \begin{align}
    \int_\R J_\eps(\zeta-\xi)(\chi(b;\xi)-\chi(a;\xi))\,d\xi &=
    J_\eps(\zeta-\theta)(b-a), \label{eq:theta}\\ 
    \int_\R
    J_\eps(\zeta-\xi)(b-\xi)(\chi(b;\xi)-\chi(a;\xi))\,d\xi &=
    \frac{1}{2}J_\eps(\zeta-\tau)(b-a)^2.\label{eq:tau}
  \end{align}
  Furthermore, whenever $a \neq b$:
  \begin{itemize}
   \item[(i)] 
    \begin{align*}
      J_\eps(\zeta-\theta) &= \frac{1}{b-a}\int_a^b J_\eps(\zeta-\xi)\,d\xi;\\ 
      J_\eps(\zeta-\tau) &= \frac{2}{(b-a)^2}\int_a^b J_\eps(\zeta-\xi)(b-\xi)\,d\xi;
    \end{align*}
   \item[(ii)]
    \begin{displaymath}
      (J_\eps(\zeta-\theta)-J_\eps(\zeta-\tau))(b-a) 
      = \frac{1}{b-a} \int_a^b J_\eps(\zeta-\xi)(2\xi-(b+a))\,d\xi;
    \end{displaymath}
   \item[(iii)]
      \begin{displaymath}
       \qquad J_\eps(\zeta-\tau)-J_\eps(\zeta-a) = \frac{2}{(b-a)^2}\Bigl(\int_a^b
        (J_\eps(\zeta-\xi)-J_\eps(\zeta-a))(b-\xi)\,d\xi \Bigr);
      \end{displaymath}
   \item[(iv)]
      \begin{multline*}
	\qquad \, J_\eps(\zeta-\theta_\eps(a,b,\zeta))-J_\eps(\zeta-\tau_\eps(a,b,\zeta)) 
	= \\ -\left(J_\eps(\zeta-\theta_\eps(b,a,\zeta))-
        J_\eps(\zeta-\tau_\eps(b,a,\zeta))\right).
      \end{multline*}
  \end{itemize}
\end{lemma}
\begin{proof}
  To prove \eqref{eq:tau}, note that
  \begin{equation*}
    \int_\R J_\eps(\zeta-\xi)(b-\xi)(\chi(b;\xi)-\chi(a;\xi))\,d\xi =
    \int_a^b J_\eps(\zeta-\xi)(b-\xi)\,d\xi. 
  \end{equation*}
  By the mean value theorem there exists a $\tau$ between $a$ and $b$
  such that
  \begin{equation*}
    \int_a^b J_\eps(\zeta-\xi)(b-\xi)\,d\xi =
    J_\eps(\zeta-\tau)\int_a^b (b-\xi)\,d\xi. 
  \end{equation*}
  Equation \eqref{eq:theta} follows in a similar way. The proof
  of (i) is immediate. Let us prove (ii). By (i)
  \begin{equation*}
    (J_\eps(\zeta-\theta)-J_\eps(\zeta-\tau))(b-a) 
    = \int_a^b J_\eps(\zeta-\xi)\left(1 - 2\frac{b-\xi}{b-a}\right)\,d\xi. 
  \end{equation*}
  It remains to observe that
  \begin{equation*}
    1 - 2\frac{b-\xi}{b-a} = \frac{2\xi-(b+a)}{b-a}.
  \end{equation*}
  To prove (iii), note that
  \begin{equation*}
    J_\eps(\zeta-a)(b-a)^2 = 2\int_a^b J_\eps(\zeta-a)(b-\xi)\,d\xi.
  \end{equation*}
  Hence (iii) follows by (i). To prove (iv), observe that the
  expression on the right-hand side of (ii) is symmetric in $a$ and $b$.
\end{proof}

The next result can be viewed as a discrete counterpart 
of the the chain rule, enabling us to write the 
nonlinear term $n_A^\Dx$, properly regularized, on a form that resembles a 
parabolic dissipation term like \eqref{eq:pardiss}.

\begin{lemma}\label{Lemma:TauThetaDefAndProperties}
  With the notation of 
  Lemma~\ref{lemma:KineticSemiDiscEquATransformed}, for 
  each $1 \le i \le d$, let
  \begin{equation*}
    \tau_{\Dx,i}^+ = \tau_\eps(A(u_\Dx),S_{\Dx_i} A(u_\Dx),\zeta),
    \quad 
    \tau_{\Dx,i}^- = \tau_\eps(A(u_\Dx),S_{-\Dx_i}A(u_\Dx),\zeta)
  \end{equation*}
  and
  \begin{equation*}
    \theta_{\Dx,i}^+ = \theta_\eps(A(u_\Dx),S_{\Dx_i} A(u_\Dx),\zeta), 
    \quad
    \theta_{\Dx,i}^- = \theta_\eps(A(u_\Dx),S_{-\Dx_i}A(u_\Dx),\zeta),
  \end{equation*}
  where $\tau_\eps,\theta_\eps$ is defined in 
  Lemma~\ref{lemma:DiscreteChainRule2Order}. Then
  \begin{itemize}
  \item[(i)]
    \begin{multline*}
      \qquad \, n_A^\Dx \star J_\eps(t,x,\zeta)
      = \frac{1}{2}\sum_{i=1}^d J_\eps(\zeta-\tau_{\Dx,i}^+)(\Dp^i A(u_\Dx))^2 \\
      + \frac{1}{2}\sum_{i=1}^d
      J_\eps(\zeta-\tau_{\Dx,i}^-)(\Dm^i A(u_\Dx))^2;
    \end{multline*}
  \item[(ii)] for $1 \le i \le d$,
    \begin{align*}
	\Dp^i \chi_\eps(A(u_\Dx);\zeta) &=
        J_\eps(\zeta-\theta_{\Dx,i}^+)\Dp^i(A(u_\Dx)),\\ 
	\Dm^i \chi_\eps(A(u_\Dx);\zeta) &=
        J_\eps(\zeta-\theta_{\Dx,i}^-)\Dm^i(A(u_\Dx)).
      \end{align*}
    \end{itemize}
\end{lemma}
\begin{proof}
  By the definition \eqref{eq:nadxdef} of $n_A^\Dx$, recalling that
  $S_y$ commutes with function evaluation,
  \begin{align*}
    &n_A^\Dx \star J_\eps (t,x,\zeta) \\
    &= \sum_{i=1}^d\frac{1}{\Dx^2}\int_\R J_\eps(\zeta-\xi)(S_{\Dx_i}
    A(u_\Dx)-\xi)(\chi(S_{\Dx_i}
    A(u_\Dx);\xi)-\chi(A(u_\Dx);\xi))\,d\xi 
    \\
    &+ \sum_{i=1}^d\frac{1}{\Dx^2}\int_\R
    J_\eps(\zeta-\xi)(S_{-\Dx_i}A(u_\Dx)-\xi)
    (\chi(S_{-\Dx_i}A(u_\Dx);\xi)-\chi(A(u_\Dx);\xi)) 
    \,d\xi.
  \end{align*}
  Hence (i) follows by Lemma~\ref{lemma:DiscreteChainRule2Order}. 
  To prove (ii) note that by Lemma~\ref{lemma:DiscreteChainRule2Order},
  \begin{align*}
    \Dp^i \int_\R &\chi(A(u_\Dx);\xi)J_\eps(\zeta-\xi)\,d\xi \\
    &= \frac{1}{\Dx} \int_\R J_\eps(\zeta-\xi)
    (\chi(S_{\Dx_i} A(u_\Dx);\xi) - \chi(A(u_\Dx);\xi))\,d\xi \\
    &= J_\eps(\zeta-\theta_{\Dx,i}^+)\Dp^i(A(u_\Dx)).
  \end{align*}
  The same argument applies to $\theta^-_{\Dx,i}$.
\end{proof}

We have now come to the key result of this subsection, 
namely a lower bound on the discrete 
dissipation term \eqref{eq:discretediss}.

\begin{lemma}\label{lemma:squareTermEst}
  Let $E_{\Dx,\eps,r,r_0}$ be defined in
  Lemma~\ref{lemma:ContractionLemma}. Then
  \begin{equation*}
    E_{\Dx,\eps,r,r_0} \ge \sum_{k = 1}^2(R_k^+ + R_k^-) \text{
      everywhere in $(r_0,T-r_0) \times \R^d \times \R$}, 
  \end{equation*}
  for all positive numbers $\Dx$, $\eps$, $r$, and $r_0$, where
  \begin{align*}
    R_1^+(\zeta) &=
    \sum_{i=1}^d((J_\eps(\zeta-\tau^+_{\Dx,i})-J_\eps(\zeta-\theta^+_{\Dx,i}))\Dp^i
    A(u_\Dx) \star J_{r_0}  \otimes J_r)\partial_{x_i}
    \rho_{\eps,r,r_0}, \\ 
    R_1^-(\zeta) &=
    \sum_{i=1}^d((J_\eps(\zeta-\tau^-_{\Dx,i})-J_\eps(\zeta-\theta^-_{\Dx,i}))\Dm^i
    A(u_\Dx)) \star J_{r_0}  \otimes J_r)\partial_{x_i}
    \rho_{\eps,r,r_0}, \\ 
    R_2^+(\zeta) &=
    \frac{1}{2}\sum_{i=1}^d\Bigl[
    \left(\left(J_\eps(\zeta-A(u_\Dx))-J_\eps(\zeta-A(\tau^+_{\Dx,i}))\right)
      \star J_{r_0} \otimes J_r\right)\\  
    &\hphantom{XXXXXXXXXXXXXX} \times 
    \left(J_\eps(\zeta-A(u))(\partial_{x_i}A(u))^2 \star J_{r_0}
      \otimes J_r\right)\Bigr], \\ 
    R_2^-(\zeta) &=
    \frac{1}{2}\sum_{i=1}^d\Bigl[
    \left(\left(J_\eps(\zeta-A(u_\Dx))-J_\eps(\zeta-A(\tau^-_{\Dx,i}))\right)
      \star J_{r_0}  \otimes J_r\right)\\ 
    & \hphantom{XXXXXXXXXXXXXX}\times
    \left(J_\eps(\zeta-A(u))(\partial_{x_i}A(u))^2 \star J_{r_0}
      \otimes J_r\right)\Bigr].
  \end{align*}
\end{lemma}
\begin{proof}
  By Lemma~\ref{Lemma:TauThetaDefAndProperties},
  \begin{align*}
    (J_\eps&(\zeta-A(u)) \star J_{r_0}  \otimes J_r)n^\Dx_{A,\eps,r,r_0} \\
    &= \frac{1}{2}\sum_{i=1}^d(J_\eps(\zeta-A(u)) \star J_{r_0}
    \otimes J_r)(J_\eps(\zeta-\tau_{\Dx,i}^+)(\Dp^i A(u_\Dx))^2 \star
    J_{r_0}  \otimes J_r) \\ 
    & \quad + \frac{1}{2}\sum_{i=1}^d(J_\eps(\zeta-A(u)) \star J_{r_0}
    \otimes J_r)(J_\eps(\zeta-\tau_{\Dx,i}^-)(\Dm^i A(u_\Dx))^2 \star
    J_{r_0}  \otimes J_r) 
    \\
    &=: \term_1^+ + \term_1^-.
  \end{align*}
  Observe that
  \begin{equation*}
    \partial_{x_i} \rho_{\eps,r,r_0} = \partial_{x_i}
    (\chi_\eps(A(u);\zeta) \star J_{r_0} \otimes J_r) 
    = J_\eps(\zeta-A(u))\partial_{x_i} A(u) \star J_{r_0}  \otimes J_r.
  \end{equation*}
  Using Lemma \ref{Lemma:TauThetaDefAndProperties} 
  once more gives
  \begin{align*}
    (\Dp^i + \Dm^i)\rho_{\eps,r,r_0}^\Dx 
    & = (J_\eps(\zeta-\theta^+_{\Dx,i})\Dp^i A(u_\Dx) 
    \star J_{r_0}  \otimes J_r 
    \\ & \qquad 
    + J_\eps(\zeta-\theta^-_{\Dx,i})\Dm^i A(u_\Dx)) \star J_{r_0}
    \otimes J_r. 
  \end{align*}
  Hence,
  \begin{align*}
    \nabla &\rho_{\eps,r,r_0} \cdot (\Dp + \Dm)\rho_{\eps,r,r_0}^\Dx  
    \\ & = \sum_{i=1}^d(J_\eps(\zeta-A(u))\partial_{x_i} A(u) \star J_{r_0}
    \otimes J_r)(J_\eps(\zeta-\theta^+_{\Dx,i})\Dp^i A(u_\Dx) \star
    J_{r_0}  \otimes J_r)
    \\ & \quad + \sum_{i=1}^d(J_\eps(\zeta-A(u))\partial_{x_i} A(u)
    \star J_{r_0} \otimes
    J_r)(J_\eps(\zeta-\theta^-_{\Dx,i})\Dm^i A(u_\Dx)) \star
    J_{r_0} \otimes J_r).
  \end{align*}
  Adding and subtracting we obtain
  \begin{equation*}
    -\nabla \rho_{\eps,r,r_0} \cdot (\Dp + \Dm)\rho_{\eps,r,r_0}^\Dx =
    \term_2^+ + \term_2^- + R_1^+ + R_1^-, 
  \end{equation*}
  where
  \begin{align*}
    \term_2^+ &= -\sum_{i=1}^d(J_\eps(\zeta-A(u))\partial_{x_i} A(u) \star J_{r_0} 
    \otimes J_r)(J_\eps(\zeta-\tau^+_{\Dx,i})\Dp^i A(u_\Dx) \star J_{r_0}  \otimes J_r), \\
    \term_2^- &= -\sum_{i=1}^d(J_\eps(\zeta-A(u))\partial_{x_i}
    A(u) \star J_{r_0} \otimes
    J_r)(J_\eps(\zeta-\tau^-_{\Dx,i})\Dm^i A(u_\Dx)) \star
    J_{r_0} \otimes J_r).
  \end{align*}
  For each $1 \le i \le d$,
  \begin{align*}
    J_\eps(\zeta-A(u_\Dx))
    &= \frac{1}{2}\left(J_\eps(\zeta-A(u_\Dx))-J_\eps(\zeta-A(\tau^+_{\Dx,i}))\right) \\
    &\qquad + \frac{1}{2}\left(J_\eps(\zeta-A(u_\Dx))-J_\eps(\zeta-A(\tau^-_{\Dx,i}))\right) \\
    &\qquad + \frac{1}{2}\left(J_\eps(\zeta-A(\tau^+_{\Dx,i})) 
    + J_\eps(\zeta-A(\tau^-_{\Dx,i}))\right).
  \end{align*}
  It follows that
  \begin{equation*}
    (J_\eps(\zeta-A(u_\Dx)) \star J_{r_0}  \otimes J_r)n_{A,\eps,r,r_0}
    = \term_3^+ + \term_3^- + R_2^+ + R_2^-,
  \end{equation*}
  where
  \begin{align*}
    \term_3^+ &= \frac{1}{2}\sum_{i=1}^d(J_\eps(\zeta-A(\tau^+_{\Dx,i}))
    \star J_{r_0}  \otimes
    J_r)(J_\eps(\zeta-A(u))(\partial_{x_i}A(u))^2 \star J_{r_0}
    \otimes J_r), 
    \\
    \term_3^- &=
    \frac{1}{2}\sum_{i=1}^d(J_\eps(\zeta-A(\tau^-_{\Dx,i}))
    \star J_{r_0} \otimes
    J_r)(J_\eps(\zeta-A(u))(\partial_{x_i}A(u))^2 \star J_{r_0}
    \otimes J_r).
  \end{align*}
  Note that
  \begin{equation*}
    E_{\Dx,\eps,r,r_0} = \sum_{k = 1}^3 (\term_k^+ + \term_k^-) + \sum_{k =
      1}^2 (R_k^+ + R_k^-).  
  \end{equation*}
  Now,
  \begin{align*}
    \sum_{k = 1}^3 \term_k^+
    &= \frac{1}{2}\sum_{i=1}^d(J_\eps(\zeta-A(u)) \star J_{r_0}
    \otimes J_r)((J_\eps(\zeta-\tau_{\Dx,i}^+)(\Dp^i A(u_\Dx))^2 \star
    J_{r_0}  \otimes J_r) 
    \\
    &\quad -\sum_{i=1}^d(J_\eps(\zeta-A(u))\partial_{x_i} A(u) \star
    J_{r_0}  \otimes J_r)(J_\eps(\zeta-\tau^+_{\Dx,i})\Dp^i A(u_\Dx)
    \star J_{r_0}  \otimes J_r) 
    \\
    &\quad +\frac{1}{2}\sum_{i = 1}^d
    (J_\eps(\zeta-A(\tau^+_{\Dx,i})) \star J_{r_0} \otimes
    J_r)(J_\eps(\zeta-A(u))(\partial_{x_i} A(u))^2 \star J_{r_0}
    \otimes J_r)\\
    &= \frac{1}{2}\sum_{i = 1}^d
    J_\eps(\zeta-A(u))J_\eps(\zeta-A(\tau^+_{\Dx,i}))
    \left(\partial_{x_i} A(u)-\Dp^i A(u_\Dx)\right)^2 
    \\
    &\hphantom{XXXXXXXXXXXXXXX}\qquad \qquad\conv{u,u_\Dx}\!\!\! J_{r_0}
    \otimes J_r \otimes J_{r_0}  \otimes J_r 
    \\
    & \ge 0.
  \end{align*}
  The obtained inequality holds for 
  all $(t,x,\zeta) \in (r_0,T-r_0) \times \R^d \times \R$. 
  Similarly,
  \begin{equation*}
    \sum_{i = 1}^3 \term_i^-(t,x,\zeta) \ge 0 \text{ for all $(t,x,\zeta)
      \in (r_0,T-r_0) \times \R^d \times \R$.} 
  \end{equation*}
  This concludes the proof of the lemma.
\end{proof}

\subsection{Bounding error terms}
We are going to estimate a series of 
``unwanted'' terms coming from Lemmas \ref{lemma:ContractionLemma} 
and \ref{lemma:squareTermEst}. To this end, we will need to gather three 
technical lemmas, the first one being a simple application 
of Young's inequality for convolutions.

\begin{lemma}\label{lemma:doubleConvolutionEstimate}
  Let $\psi: \R^2 \rightarrow \R$ be a measurable function, and 
  $u,v:\R^d \rightarrow \R$ be measurable functions satisfying
  \begin{equation*}
    \Bigl|\psi(u(x_1),v(x_2))\Bigr| \le K_1(x_1)K_2(x_2) 
    \qquad (x_1,x_2 \in \R^d),
  \end{equation*}
  for some $K_1 \in L^p(\R^d), 1 \le p \le \infty$, 
  and $K_2 \in L^1(\R^d)$. Then
  \begin{equation*}
    \Bigl\|\psi(u,v) \conv{u,v} f \otimes g\Bigr\|_{L^1(\R^d)} \le
    \norm{K_1}_{L^p(\R^d)}\norm{K_2}_{L^1(\R^d)}
    \norm{f}_{L^q(\R^d)}\norm{g}_{L^1(\R^d)},
  \end{equation*}
  for any $g \in L^1(\R)$ and $f \in L^q(\R)$ 
  where $p^{-1} + q^{-1} = 1$. 
  
  If $\psi \in L^\infty(\R^2)$, then
  \begin{equation}\label{eq:lemma:doubleConvolutionEstimateInftyBound}
    \Bigl\|\psi(u,v) \conv{u,v} f \otimes g\Bigr\|_{L^\infty(\R^d)}
    \le
    \norm{\psi}_{L^\infty(\R^2)}\norm{f}_{L^1(\R^d)}\norm{g}_{L^1(\R^d)}. 
  \end{equation}
\end{lemma}

\begin{proof}
  Observe that
  \begin{align*}
    &\Bigl\|\psi(u,v) \conv{u,v} f \otimes g\Bigl\|_{L^1(\R^d)}\\
    &\hphantom{XX}\le \iiint_{\R^d \times \R^d \times \R^d}
    K_1(y_1)K_2(y_2)\abs{f(x-y_1)}\abs{g(x-y_2)}\,dy_1dy_2dx
    \\ &\hphantom{XX}= \norm{(K_1 \star \abs{f})(K_2 \star
      \abs{g})}_{L^1(\R^d)}.
  \end{align*}
  By H\"older's inequality,
  \begin{equation*}
    \norm{(K_1 \star \abs{f})(K_2 \star \abs{g})}_{L^1(\R^d)} \le
    \norm{K_1 \star \abs{f}}_{L^\infty(\R^d)} \norm{K_2 \star
      \abs{g}}_{L^1(\R^d)}. 
  \end{equation*}
  By Young's inequality for convolutions, 
  $\norm{K_1 \star  \abs{f}}_{L^\infty(\R^d)} \le
  \norm{K_1}_{L^p(\R^d)}\norm{f}_{L^q(\R^d)}$ 
  and $\norm{K_2 \star \abs{g}}_{L^1(\R^d)} \le
  \norm{K_2}_{L^1(\R^d)}\norm{g}_{L^1(\R^d)}$. 
  Equation~\eqref{eq:lemma:doubleConvolutionEstimateInftyBound} 
  follows, since
  \begin{equation*}
    \Bigl|\int_{\R^d}\int_{\R^d}
    \psi(u(y_1),v(y_2))f(x-y_1)g(x-y_2)\,dy_1dy_2\Bigr|
    \le \norm{\psi}_{L^\infty(\R^2)}\norm{f}_{L^1(\R^d)}\norm{g}_{L^1(\R^d)}. 
  \end{equation*}
\end{proof}

The next lemma is at the heart of the matter, permitting us to 
estimate some terms involving convolutions against 
approximate delta functions.

\begin{lemma}\label{lemma:tauthetaaIntDiff}
  For real numbers $a$ and $b$, let $\tau=\tau_\eps(a,b,\zeta)$ and
  $\theta=\theta_\eps(a,b,\zeta)$ be as in
  Lemma~\ref{lemma:DiscreteChainRule2Order}.
  \begin{itemize}
  \item[(i)] For $f \in L^1_\mathrm{loc}(\R)$, define
    \begin{equation*}
      \term^1_\eps(f) = \int_\R(J_\eps(\zeta-\tau)-J_\eps(\zeta-a))f(\zeta)\,d\zeta.
    \end{equation*}
    Then
    \begin{equation*}
      \abs{\term^1_\eps(f)} \le \Bigl|\int_a^b \abs{\partial_\xi (f \star
        J_\eps)(\xi)}\,d\xi\Bigr|. 
    \end{equation*}
  \item[(ii)] For $f \in L^\infty(\R)$, define
    \begin{equation*}
      \term^2_\eps(f) = \int_\R
      (J_\eps(\zeta-\theta)-J_\eps(\zeta-\tau))(b-a)f(\zeta)\,d\zeta. 
    \end{equation*}
    Then
    \begin{equation*}
      \abs{\term^2_\eps(f)} \le \frac{1}{2}\norm{f}_{L^\infty(\R)}\abs{b-a}.
     \end{equation*}
  \item[(iii)] Suppose $\seq{f_\eps}_{\eps > 0} \subset
    W^{1,1}_\mathrm{loc}(\R)$ and assume that there exists 
    $f \in W^{1,1}_\mathrm{loc}(\R)$ such 
    that $f_\eps \rightarrow f$ in $W^{1,1}_\mathrm{loc}(\R)$. Then
    \begin{equation*}
      \Bigl|\lim_{\eps \downarrow 0}\term^1_\eps(f_\eps)\Bigr|
      \le \Bigl|\int_a^b \abs{f'(\xi)}\,d\xi\Bigr|. 
    \end{equation*}
  \end{itemize}
\end{lemma}
\begin{proof}
  Assume that $a < b$. By Lemma~\ref{lemma:DiscreteChainRule2Order}
  and Fubini's Theorem,
  \begin{equation*}
    \term^1_\eps(f) = \frac{2}{(b-a)^2} \int_a^b \left((f \star
      J_\eps)(\xi)-(f \star J_\eps)(a)\right)(b-\xi)\,d\xi. 
  \end{equation*}
  Since $\partial_\xi(b-\xi)^2 = -2(b-\xi)$, integration by parts yields
  \begin{equation}\label{eq:T1Representation}
    \term^1_\eps(f) = \int_a^b \partial_\xi(f \star
    J_\eps)(\xi)\frac{(b-\xi)^2}{(b-a)^2} \,d\xi. 
  \end{equation}
  Then statement (i) follows, since
  \begin{equation}\label{eq:lemma:tauthetaaIntDiffEstOnIntegrand}
    \frac{(b-\xi)^2}{(b-a)^2} \le 1, \text{ whenever $a \le \xi \le b$.}
  \end{equation}
  
  By Lemma~\ref{lemma:DiscreteChainRule2Order},
  \begin{equation*}
    \term^2_\eps(f) = \frac{1}{b-a}\int_a^b (f \star J_\eps)(\xi)(2\xi-(a+b))\,d\xi.
  \end{equation*}
  As $\norm{f \star J_\eps}_{L^\infty(\R)} \le
  \norm{f}_{L^\infty(\R)}$, we may conclude that
  \begin{equation*}
    \term^2_\eps(f) \le \frac{1}{b-a}\int_a^b
    \norm{f}_{L^\infty(\R)}\abs{(2\xi-(a+b))}\,d\xi. 
  \end{equation*}
  This implies statement (ii), because 
  \begin{equation*}
    \int_a^b \abs{2\xi-(a+b)}\,d\xi = \frac{1}{2}(b-a)^2.
  \end{equation*}
  
  Finally, we establish statement  (iii). By the triangle inequality 
  and Young's inequality for convolutions,
  \begin{displaymath}
   \norm{f_\varepsilon' \star J_\varepsilon - f'}_{L^1(V)} 
   \leq \norm{f_\varepsilon' - f'}_{L^1(V)} + \norm{f' \star J_\varepsilon - f'}_{L^1(V)} 
  \end{displaymath}
  for any compact $V \subset \R$. Hence $(f_\eps' \star J_\eps) \rightarrow
  f'$ in $L^1_\mathrm{loc}(\R)$ as $\varepsilon \downarrow 0$. By \eqref{eq:T1Representation} and \eqref{eq:lemma:tauthetaaIntDiffEstOnIntegrand} it follows that 
  \begin{equation*}
    \lim_{\eps \downarrow 0} \term^1_\eps(f_\eps) = 
    \lim_{\eps \downarrow 0}\int_a^b (f_\eps' \star J_\eps)(\xi)\frac{(b-\xi)^2}{(b-a)^2} \,d\xi 
    = \int_a^b f'(\xi)\frac{(b-\xi)^2}{(b-a)^2} \,d\xi.
  \end{equation*} 
  The estimate follows thanks to 
  \eqref{eq:lemma:tauthetaaIntDiffEstOnIntegrand}.
\end{proof}

We need one more lemma bounding some specific convolution integrals.

\begin{lemma}\label{lemma:MollifiedChiProd}
  Suppose $f \in C(\R)$ and let $R^f_\eps:\R^2 \rightarrow \R$ be
  defined by \eqref{eq:RfDefinition}. Then
  \begin{equation}\label{eq:MollificationCommutationErrTerm}
    R_\eps^f(u,\zeta) = \int_0^u (f(\sigma)-f(\zeta))J_\eps(\sigma-\zeta)\,d\sigma,
  \end{equation}
  for all $u,\zeta \in \R$. Furthermore, if $f$ is 
  Lipschitz continuous, then
    \begin{equation}\label{eq:mollidiedchiest}
      \int_\R \abs{R_\eps^f(u,\zeta)}\,d\zeta \le \eps \norm{f}_\Lip \abs{u}.
    \end{equation}
  Suppose $A' \ge \eta >0$ and let $g$ be defined by $g \circ A = f$. 
  For $a,b \in \R$, let
    \begin{equation*}
      Z(a,b) = \int_\R \signe{\zeta-a}R^{g'}_{\eps}(b;\zeta)\,d\zeta.
    \end{equation*}
    Then
    \begin{equation}\label{eq:IntByPartsBoundOnRGDer}
      \abs{Z(a,b)} \le 4\norm{f}_\Lip\frac{\eps}{\eta}.
    \end{equation}
\end{lemma}
\begin{proof}
  Observe that
  \begin{align*}
    & \int_\R f(\sigma)\chi(u;\sigma)J_\eps(\zeta-\sigma)\,d\sigma 
    \\ & \qquad 
    = \underbrace{\int_\R
    (f(\sigma)-f(\zeta))\chi(u;\sigma)
    J_\eps(\zeta-\sigma)\,d\sigma}_{R^f_\eps(u,\zeta)} 
    + f(\zeta)\chi_\eps(u;\zeta).
  \end{align*}
  Let $q'(\sigma) =(f(\sigma)-f(\zeta))J_\eps(\zeta-\sigma)$. 
  Equation~\eqref{eq:MollificationCommutationErrTerm} follows, since
  \begin{equation*}
    \int_\R
    (f(\sigma)-f(\zeta))\chi(u;\sigma)J_\eps(\zeta-\sigma)\,d\sigma =
    q(u)-q(0) = \int_0^u q'(\sigma)\,d\sigma. 
  \end{equation*}
  
  To prove \eqref{eq:mollidiedchiest} observe that
  \begin{align*}
    \int_\R \abs{R^f_\eps(u,\zeta)}\,d\zeta &\le \int_\R 
    \abs{\chi(u;\sigma)}\left(\int_\R
      \abs{f(\sigma)-f(\zeta)}J_\eps(\zeta-\sigma)\,d\zeta  
    \right) d\sigma.
  \end{align*}
  The result follows as
  \begin{equation*}
    \int_\R \abs{f(\sigma)-f(\zeta)}J_\eps(\zeta-\sigma)\,d\sigma \le
    \norm{f}_\Lip\eps. 
  \end{equation*}
   
  Let us prove \eqref{eq:IntByPartsBoundOnRGDer}. Take
  \begin{equation*}
    H^{g'}_\eps(b;\zeta) = \int_0^\zeta R^{g'}_{\eps}(b;\sigma)\,d\sigma.
  \end{equation*}
  Integration by parts yields
  \begin{equation*}
    Z(a,b) = -2\int_\R J_\eps(\zeta-a)H^{g'}_{\eps}(b;\zeta)\,d\zeta 
    = -(H^{g'}_{\eps}(b;\cdot) \star J_\eps)(a).
  \end{equation*}
  By \eqref{eq:MollificationCommutationErrTerm},
  \begin{equation*}
    H^{g'}_\eps(b;a) = \int_0^a\int_0^b(g'(\omega)-g'(\sigma))
    J_\eps(\omega -\sigma)\,d\omega \,d\sigma.
  \end{equation*}
  Due to the symmetry of $J_\eps$,
  \begin{align*}
    H^{g'}_\eps(b;a) & 
    = \int_0^a\int_0^b g'(\omega)J_\eps(\omega
    -\sigma)\,d\omega \,d\sigma 
    \\ &\qquad- \int_0^a\int_0^b g'(\sigma)J_\eps(\omega -\sigma)\,d\omega
    \,d\sigma 
    \\ & =\int_0^a\int_0^b g'(\omega)J_\eps(\omega -\sigma)\,d\omega
    \,d\sigma 
    \\
    &\qquad -\int_0^b\int_0^a g'(\omega)J_\eps(\omega -\sigma)\,d\omega
    \,d\sigma 
    \\
    &=\int_\R g'(\omega)\chi(b;\omega)\Bigl(\int_0^a J_\eps(\omega
      -\sigma)\,d\sigma\Bigr)\,d\omega 
    \\
    &\qquad - \int_\R g'(\omega)\chi(a;\omega)\Bigl(\int_0^bJ_\eps(\omega
      -\sigma)\,d\sigma \Bigr)\,d\omega.
  \end{align*}
  Note that
  \begin{equation*}
    \int_0^a J_\eps(\omega -\sigma)\,d\sigma = \int_\R
    \chi(a;\sigma)J_\eps(\omega-\sigma)\,d\sigma 
    = \chi_\eps(a;\omega). 
  \end{equation*}
  Hence,
  \begin{equation*}
    H^{g'}_\eps(b;a) = \int_\R
    g'(\omega)\left(\chi(b;\omega)\chi_\eps(a;\omega)-\chi(a;\omega)
      \chi_\eps(b;\omega)\right)\,d\omega. 
  \end{equation*}
 Set
  \begin{equation*}
    \lambda(a,b;\omega) :=
    \chi(b;\omega)\chi_\eps(a;\omega)-\chi(a;\omega)
    \chi_\eps(b;\omega).
  \end{equation*}
  To find the support of $\lambda(a,b;\omega)$ 
  we first observe that
  $\lambda(a,b;\omega) 
  = -\lambda(b,a;\omega)$. 
  This reduces the situation to the following cases:
  \begin{equation*}
    \begin{cases}
      0 \le a \le b: & \abs{\lambda(a,b;\omega)} \le
      \car{\abs{a-\omega} \le \eps},
      \\
      b \le a \le 0: & \abs{\lambda(a,b;\omega)} \le
      \car{\abs{a-\omega} \le \eps},
      \\
      a \le 0 \le b: & \abs{\lambda(a,b;\omega)} \le \car{\abs{\omega}
        \le \eps}.
    \end{cases}
  \end{equation*}
  It thus follows that
  \begin{equation*}
    \abs{H^{g'}_{\eps}(b,a)} \le 2\norm{g'}_{\infty}\eps.
  \end{equation*}
  Statement \eqref{eq:IntByPartsBoundOnRGDer} follows as 
  $g'(A(z))A'(z) = f'(z)$, which implies
  $\norm{g'}_{\infty} \le \norm{f}_\Lip\eta^{-1}$.
\end{proof}

We have now the tools needed to start estimating 
the error terms in Lemmas \ref{lemma:ContractionLemma} 
and \ref{lemma:squareTermEst}, starting with those 
in Lemma \ref{lemma:squareTermEst}.

\begin{estimate} \label{est:1}
Let $R_1^\pm$ be defined in Lemma~\ref{lemma:squareTermEst}. 
Then there exists a constant $C = C(d,J)$ such that 
  \begin{equation*}
    \norm{\int_\R R_1^+(\zeta) + R_1^-(\zeta) \,d\zeta}_{L^1(\Pi_T^{r_0})} 
    \le C\frac{\Dx}{r^2}\left(1 + \frac{\Dx}{r}\right)\norm{\Dm A(u_\Dx)}_{L^1(\Pi_T;\R^d)}.
  \end{equation*}
\end{estimate}
\begin{proof}
  Let us first make an observation regarding the similarity of these terms. 
  By statement (iv) of Lemma~\ref{lemma:DiscreteChainRule2Order},
  recalling also the definition of 
  $\theta^{\pm}_{\Dx,i}$ and $\tau^{\pm}_{\Dx,i}$ in
  Lemma~\ref{Lemma:TauThetaDefAndProperties},
   \begin{align*}
    &S_{\Dx_i}\left(J_\eps(\zeta-\tau^-_{\Dx,i})
    -J_\eps(\zeta-\theta^-_{\Dx,i})\right)
    \\ & \quad =
    J_\eps(\zeta-\tau_\eps(S_{\Dx_i}A(u_\Dx),A(u_\Dx),\zeta))
    -J_\eps(\zeta-\theta_\eps(S_{\Dx_i}A(u_\Dx),A(u_\Dx),\zeta)) 
    \\
    & \quad =
    -\left(J_\eps(\zeta-\tau^+_{\Dx,i})-J_\eps(\zeta-\theta^+_{\Dx,i})\right).
  \end{align*}
  Recalling that $\rho_{\eps,r,r_0} =
  \chi_{\eps}(A(u);\cdot) \star J_{r_0} \otimes J_r$, which implies
  \begin{align*}
    & R_1^+(\zeta) + R_1^-(\zeta) \\
    & \quad = -\sum_{i=1}^d S_{\Dx_i}\left[\left(J_\eps(\zeta-\tau^-_{\Dx,i})-
        J_\eps(\zeta-\theta^-_{\Dx,i})\right)D^i_-(u_\Dx)\star (J_{r_0}\otimes J_r)\right]
        \partial_{x_i} \rho_{\eps,r,r_0} \\
    &\, \qquad 
    + \sum_{i=1}^d \left[\left(J_\eps(\zeta-\tau^-_{\Dx,i})-
     J_\eps(\zeta-\theta^-_{\Dx,i})\right)D_-^i(u_\Dx)\star (J_{r_0}\otimes J_r)\right]
        \partial_{x_i} \rho_{\eps,r,r_0} 
    \\ & \quad =-\sum_{i=1}^d(S_{\Dx_i}-1)
    \partial_{x_i} \rho_{\eps,r,r_0}
    \\ & \quad \qquad\qquad\qquad \times \left[\left(J_\eps(\zeta-\tau^-_{\Dx,i})
    - J_\eps(\zeta-\theta^-_{\Dx,i})\right)D_-^i(u_\Dx)\star(J_{r_0}\otimes J_r)\right] 
    \\ & \quad =-\Dx\sum_{i=1}^d
    \left(J_\eps(\zeta-\tau^-_{\Dx,i})-J_\eps(\zeta-\theta^-_{\Dx,i}))
    \Dm^i A(u_\Dx)\right) \chi_{\eps}(A(u);\zeta) 
    \\ &\hphantom{=-\Dx\sum_{i=1}^d}\qquad 
    \conv{u,u_\Dx}(J_{r_0} 
    \otimes \partial_{x_i}J_r) \otimes (J_{r_0} \otimes \Dp^iJ_r).
  \end{align*}
  By statement (ii) of Lemma~\ref{lemma:tauthetaaIntDiff},
  \begin{multline}\label{eq:estBytauthetaaIntDiffii}
    \Bigl|\int_\R R_1^+(\zeta) + R_1^-(\zeta)\,d\zeta\Bigr| \\
    \le \frac{\Dx}{2} \sum_{i = 1}^d
    \norm{\chi_{\eps}(A(u);\cdot)}_{L^\infty(\R)}\abs{\Dm^i A(u_\Dx)}
    \!\!\conv{u,u_\Dx}\!\! 
    \abs{J_{r_0} \otimes \partial_{x_i}J_r 
    \otimes J_{r_0} \otimes \Dp^iJ_r}.
  \end{multline}
  By Lemma~\ref{lemma:doubleConvolutionEstimate},
  \begin{multline*}
    \Bigl\|\int_\R R_1^+(\zeta) + R_1^-(\zeta)\,d\zeta\Bigr\|_{L^1(\Pi_T^{r_0})}
    \le \frac{\Dx}{2}\sum_{i=1}^d \norm{\Dm^i A(u_\Dx)}_{L^1(\Pi_T)} \\
    \times
    \norm{\partial_{x_i}J_r}_{L^1(\R^d)}\norm{\Dp^iJ_r}_{L^1(\R^d)}.
  \end{multline*}
  Recall that $\norm{\partial_{x_i}J_r}_{L^1(\R^d)} \le
  2\norm{J'}_\infty r^{-1}$. Note that
  \begin{align*}
    \abs{\Dp^iJ_r(x)} &= \frac{1}{\Dx}\abs{J_r(x_i + \Dx)-J_r(x_i)}
    \prod_{j \neq i} J_r(x_j)
    \\
    &\le \frac{1}{r^2}\norm{J'}_\infty\car{\abs{x_i} \le r + \Dx}
    \prod_{j \neq i} J_r(x_j).
  \end{align*}
  Hence
  \begin{equation}\label{eq:FiniteDiffofMollifierEst} 
    \norm{\Dp^iJ_r}_{L^1(\R^d)} 
    \le \frac{1}{r^2}\norm{J'}_\infty \int_\R \car{\abs{x_i} \le r + \Dx} \,dx_i
    = 2\norm{J'}_\infty\frac{1}{r}\left(1 + \frac{\Dx}{r}\right)
  \end{equation}
  The estimate follows from \eqref{eq:FiniteDiffofMollifierEst} and
  \eqref{eq:estBytauthetaaIntDiffii}. 
\end{proof}

\begin{estimate} \label{est:2}
  Let $R_2^\pm$ be defined in Lemma~\ref{lemma:squareTermEst}. 
  Then there exists a constant $C = C(d,J)$ such that 
  \begin{multline*}
    \Bigl\|\int_\R R_2^+(\zeta) +
    R_2^-(\zeta)\,d\zeta\Bigr\|_{L^1(\Pi_T^{r_0})} \\
    \le 
    C\frac{\Dx}{\eps^2\sqrt{r_0
        r^d}}\norm{\Dp A(u_\Dx)}_{L^2(\Pi_T;\R^d)}
    \norm{\nabla A(u)}_{L^2(\Pi_T;\R^d)}^2.
  \end{multline*}
\end{estimate}
\begin{proof}
  Let us consider $R_2^+$. The term $R_2^-$ is treated the same way. By
  Lemma~\ref{lemma:tauthetaaIntDiff},
  \begin{align*}
    \Bigl|\int_\R R_2^+(\zeta)\,d\zeta\Bigr|
    &\le \frac{1}{2}\sum_{i=1}^d \Bigl|\int_\R
    \left(J_\eps(\zeta-A(u_\Dx))-J_\eps(\zeta-A(\tau_{\Dx,i}^+))\right)
    J_\eps(\zeta-A(u))\,d\zeta\Bigr| 
    \\
    & \hphantom{XXXXXXXXXXXX}
    \times (\partial_{x_i}A(u))^2 \conv{u,u_\Dx}(J_{r_0} \otimes J_r)
    \otimes (J_{r_0} \otimes J_r)
    \\
    &\le \frac{1}{2}\sum_{i=1}^d
    \Bigl|\int_{A(u_\Dx)}^{S_{\Dx_i}A(u_\Dx)}\abs{\partial_\xi
      (J_\eps(\cdot-A(u))\star J_\eps(\xi))}\,d\xi\Bigr|
    \\
    & \hphantom{XXXXXXXXXXXX}\times (\partial_{x_i}A(u))^2
    \conv{u,u_\Dx}(J_{r_0} \otimes J_r) \otimes (J_{r_0} \otimes J_r).
  \end{align*}
  By Young's inequality for convolutions,
  \begin{equation*}
    \norm{J_\eps(\cdot-A(u))\star J_\eps'}_{L^\infty(\R)} \le
    \norm{J_\eps(\cdot-A(u))}_{L^\infty(\R)}\norm{J_\eps'}_{L^1(\R)}
    \le \frac{2}{\eps^2}\norm{J}_\infty\norm{J'}_\infty. 
  \end{equation*}
  Hence,
  \begin{multline*}
    \Bigl|\int_\R R_2^+(\zeta)\,d\zeta\Bigr|
    \le \frac{\Dx}{\eps^2}\norm{J}_\infty\norm{J'}_\infty\sum_{i=1}^d
    \abs{\Dp^i A(u_\Dx)}(\partial_{x_i}A(u))^2 
    \\
    \conv{u,u_\Dx}(J_{r_0} \otimes J_r) \otimes (J_{r_0} \otimes J_r).
  \end{multline*}
  Applying Lemma~\ref{lemma:doubleConvolutionEstimate}, with
  $K_1(u)=\abs{u}$, $K_2(v)=v^2$, $p=q=2$, we get
  \begin{multline*}
    \Bigl\|\int_\R R_2^+(\zeta)\,d\zeta\Bigr\|_{L^1(\Pi_T^{r_0})} 
    \le \frac{\Dx}{\eps^2}\norm{J}_\infty\norm{J'}_\infty\sum_{i=1}^d
    \norm{\Dp^i
      A(u_\Dx)}_{L^2(\Pi_T)}\norm{\partial_{x_i}A(u)}_{L^2(\Pi_T)}^2 
    \\
    \times \norm{J_{r_0} \otimes J_r}_{L^2(\R \times
      \R^d)}\norm{J_{r_0} \otimes J_r}_{L^1(\R \times \R^d)}.
  \end{multline*}
  Now
  \begin{equation*}
    \norm{J_{r_0} \otimes J_r}_{L^2(\R \times \R^d)} = 
    \norm{J_{r_0}}_{L^2(\R)}\prod_{i=1}^d \norm{J_{r}}_{L^2(\R)} 
    = \frac{1}{\sqrt{r_0r^d}}\norm{J}_{L^2(\R)}^{d+1}.
  \end{equation*}
\end{proof}
\begin{estimate} \label{est:3}
  Let $R_1^\pm$ be defined in Lemma~\ref{lemma:squareTermEst} and
  suppose $d = 1$. Then there exists a constant $C = C(J)$ such that
  \begin{multline*}
    \norm{\lim_{\eps \downarrow 0} \int_\R R_2^+(\zeta) +
      R_2^-(\zeta)\,d\zeta}_{L^1(\Pi_T^{r_0})} 
    \\
    \le C\left(\frac{\Dx}{r}\norm{f}_\Lip +
      \frac{\Dx}{r^2}\norm{A}_\Lip + \frac{\Dx}{r_0}\right)\norm{\Dp
      u_\Dx}_{L^1(\Pi_T)}.
  \end{multline*}
\end{estimate}
\begin{proof}
  We consider $R_2^+$. The $R_2^-$ term can be treated similarly. 
  Note that for $d=1$,
  \begin{equation*}
  \int_\R R_2^+(\zeta)\,d\zeta = \int_\R
    \left(J_\eps(\zeta-A(u_\Dx))-J_\eps(\zeta-A(\tau^+_\Dx))\right)
    n_{A,\eps,r,r_0}(\zeta)\,d\zeta
    \conv{u_\Dx} J_{r_0} \otimes J_r,
  \end{equation*}
  where $n_{A,\eps,r,r_0}$ is defined in
  Lemma~\ref{lemma:KineticFormATransformedMollified}. 
  The map $\zeta \mapsto n_{A,\eps,r,r_0}(t,x,\zeta)$ belongs to
  $W^{1,1}_\mathrm{loc}(\R)$ for each fixed $(t,x) \in (r_0,T-r_0)\times \R$.
  Due to Lemmas~\ref{lemma:KineticFormATransformedMollified}, \ref{lemma:MollifiedChiProd}, 
  and \ref{lemma:MollifiedChiFuncProp},
  \begin{equation*}
    \lim_{\eps\to 0}\partial_\zeta n_{A,\eps,r,r_0}(\zeta) =
    B'(\zeta)\partial_t\rho_{r,r_0}(\zeta) + g'(\zeta)\partial_x
    \rho_{r,r_0}(\zeta)-\partial_x^2 \rho_{r,r_0}(\zeta) 
  \end{equation*}
  in $L^1(\R)$ for each fixed $(t,x)$, where
  \begin{equation*}
    \rho_{r,r_0}(\zeta) = \chi(A(u);\zeta) \star J_{r_0} \otimes J_r.
  \end{equation*}
  By statement (iii) of Lemma~\ref{lemma:tauthetaaIntDiff},
  \begin{align*}
    \Bigl|\lim_{\eps \downarrow 0} \int_\R R_2^+(\zeta)\,d\zeta\Bigr|
    &\le \Bigl|\int_{A(u_\Dx)}^{S_\Dx A(u_\Dx)}
    \abs{B'(\zeta)\partial_t\rho_{r,r_0}(\zeta)}\,d\zeta\Bigr|
    \conv{u_\Dx} J_{r_0} \otimes J_r
    \\
    &+\Bigl|\int_{A(u_\Dx)}^{S_\Dx A(u_\Dx)} \abs{g'(\zeta)\partial_x
      \rho_{r,r_0}(\zeta)} \,d\zeta\Bigr|
    \conv{u_\Dx} J_{r_0} \otimes J_r \\
    &+\Bigl|\int_{A(u_\Dx)}^{S_\Dx A(u_\Dx)} \abs{\partial_x^2
      \rho_{r,r_0}(\zeta)} \,d\zeta\Bigr| \conv{u_\Dx} J_{r_0} \otimes
    J_r 
    \\
    & =: \term_1 + \term_2 +\term_3.
  \end{align*}
  We consider each term separately. Let $B(\zeta) = \xi$ or
  equivalently $A(\xi) = \zeta$. It follows that
  \begin{align*}
    \term_1 &\le \Bigl|\int_{A(u_\Dx)}^{S_\Dx A(u_\Dx)}
    \abs{B'(\zeta)\chi(A(u);\zeta)}\,d\zeta\Bigr|
    \conv{u,u_\Dx}\abs{\partial_tJ_{r_0}} \otimes J_r \otimes J_{r_0} \otimes J_r \\
    &=\Bigl|\int_{u_\Dx}^{S_\Dx u_\Dx}
    \abs{\chi(u;\xi)}\,d\xi\Bigr|
    \conv{u,u_\Dx}\abs{\partial_tJ_{r_0}} \otimes J_r \otimes J_{r_0} \otimes J_r \\
    &\le \Dx \abs{\Dp A(u_\Dx)} \conv{u,u_\Dx}\abs{\partial_tJ_{r_0}}
    \otimes J_r \otimes J_{r_0} \otimes J_r.
  \end{align*}
  By Lemma~\ref{lemma:doubleConvolutionEstimate},
  \begin{equation*}
    \norm{\term_1}_{L^1(\Pi_T^{r_0})} \le
    2\frac{\Dx}{r_0}\norm{J'}_\infty\norm{\Dp u_\Dx}_{L^1(\Pi_T)}. 
  \end{equation*}
  
  Observe that $g'(A(\xi))A'(\xi) = f'(\xi)$ and $d\zeta =
  A'(\xi)d\xi$. Hence,
  \begin{align*}
    \term_2 &\le \Bigl|\int_{A(u_\Dx)}^{S_\Dx A(u_\Dx)}
    \abs{g'(\zeta)\chi(A(u);\zeta)\Bigr| \,d\zeta}
    \conv{u,u_\Dx}J_{r_0} \otimes \abs{\partial_xJ_r} \otimes J_{r_0}
    \otimes J_r
    \\
    &=\Bigl|\int_{u_\Dx}^{S_\Dx u_\Dx} \abs{f'(\xi)\chi(u;\xi)}
    \,d\xi\Bigr|
    \conv{u,u_\Dx}J_{r_0} \otimes \abs{\partial_xJ_r} \otimes J_{r_0}
    \otimes J_r 
    \\
    &\le \Dx \norm{f}_\Lip \abs{\Dp u_\Dx} \conv{u,u_\Dx}J_{r_0}
    \otimes \abs{\partial_xJ_r} \otimes J_{r_0} \otimes J_r .
  \end{align*}
  By Lemma~\ref{lemma:doubleConvolutionEstimate},
  \begin{equation*}
    \norm{\term_2}_{L^1(\Pi_T^{r_0})} \le
    2\norm{f}_\Lip\frac{\Dx}{r}\norm{J'}_\infty\norm{\Dp
      u_\Dx}_{L^1(\Pi_T)}. 
  \end{equation*}
 
  Similarly,
  \begin{align*}
    \term_3 &\le \Bigl|\int_{A(u_\Dx)}^{S_\Dx A(u_\Dx)}
    \abs{\chi(A(u);\zeta)} \,d\zeta\Bigr|
    \conv{u,u_\Dx} J_{r_0} \otimes 
    \abs{\partial_x^2 J_r} \otimes J_{r_0} \otimes J_r \\
    &\le \Dx\abs{\Dp A(u_\Dx)} \conv{u,u_\Dx} J_{r_0} \otimes
    \abs{\partial_x^2 J_r} \otimes J_{r_0} \otimes J_r.
  \end{align*}
  By Lemma~\ref{lemma:doubleConvolutionEstimate},
  \begin{equation*}
    \norm{\term_3}_{L^1(\Pi_T^{r_0})} \le
    2\frac{\Dx}{r^2}\norm{J''}_\infty\norm{\Dp
      A(u_\Dx)}_{L^1(\Pi_T)}. 
  \end{equation*}
\end{proof}

\begin{estimate} \label{est:4}
  Let $U$ be the second term in \eqref{eq:T1}, 
  Lemma~\ref{lemma:ContractionLemma}, that is,
  \begin{displaymath}
   U = 2\int_\R \nabla \rho_{\eps,r,r_0} \cdot (2\nabla
    -(\Dp + \Dm)\rho_{\eps,r,r_0}^\Dx \,d\zeta.
  \end{displaymath}
  Then there exists a constant $C = C(d,J)$ such that
  \begin{equation*}
    \norm{U}_{L^1(\Pi_T^{r_0})}
    \le C\frac{\Dx^2}{r^3}\left(1 +
      \frac{\Dx}{r}\right)\norm{A(u_\Dx)}_{L^1([0,T];BV(\R^d))}.
  \end{equation*}
\end{estimate}

\begin{remark}
  The $BV$ norm may be replaced by the $L^1$ 
  norm at the expense of an extra factor $r^{-1}$.
\end{remark}
\begin{proof}
  Clearly,
  \begin{multline*}
    \norm{\int_\R \nabla \rho_{\eps,r,r_0} \cdot
    (2\nabla -(\Dp + \Dm)\rho_{\eps,r,r_0}^\Dx \,d\zeta}_{L^1(\Pi_T^{r_0})}
    \\
    \le \norm{\nabla \rho_{\eps,r,r_0}}_{L^\infty(\Pi_T^{r_0} \times
      \R;\R^d_\infty)}\norm{(2\nabla -(\Dp +
      \Dm)\rho_{\eps,r,r_0}^\Dx}_{L^1(\Pi_T^{r_0} \times \R;\R^d_1)}.
  \end{multline*}
  By Young's inequality for convolutions,
  \begin{align*}
    \norm{\partial_{x_i}\rho_{\eps,r,r_0}}_{L^\infty(\Pi_T^{r_0} \times \R)}
    &\le \norm{\chi(A(u);\cdot)}_{L^\infty(\Pi_T \times
      \R)}\norm{J_\eps \otimes \partial_{x_i}J_r \otimes
      J_{r_0}}_{L^1(\R \times \R^d \times \R)} \notag
    \\
    &\le \norm{\partial_{x_i}J_r}_{L^1(\R^d)} \le 2\norm{J'}_\infty r^{-1}.
  \end{align*}
  We have
  \begin{equation*}
    (2\partial_{x_i} -(\Dp^i + \Dm^i))\rho_{\eps,r,r_0}^\Dx =
    \chi(A(u_\Dx);\cdot) \star  J_\eps \otimes 
    (2\partial_{x_i} -(\Dp^i + \Dm^i))J_r \otimes J_{r_0}. 
  \end{equation*}
  Using Taylor expansions with remainders,
  \begin{align*}
    \left((\Dp^i + \Dm^i)-2\partial_{x_i}\right)J_r(x)
    &= \frac{1}{2\Dx}
    \int_0^{\Dx}(z-\Dx)^2 \partial_{x_i}^3J_r(x_i+z)\,dz
    \,\prod_{j \neq i}J_r(x_j)
    \\
    &\qquad +  \frac{1}{2\Dx}
    \int_{-\Dx}^0(z+\Dx)^2 \partial_{x_i}^3J_r(x_i+z)\,dz
    \,\prod_{j \neq i}J_r(x_j)
    \\
    &=: \partial_{x_i}(\varphi^i_1(x) + \varphi^i_2(x)),
  \end{align*}
  see for instance \cite[p.~25]{KRS2014}. Hence
  \begin{equation*}
    (2\partial_{x_i} -(\Dp^i + \Dm^i))\rho_{\eps,r,r_0}^\Dx
    = \partial_{x_i}\left(\chi(A(u_\Dx);\cdot) \star J_\eps \otimes
      (\varphi^i_1 + \varphi^i_2) \otimes J_{r_0}\right). 
  \end{equation*}
  By Young's inequality for convolutions 
  \begin{multline*}
    \norm{(2\partial_{x_i} -(\Dp^i +
      \Dm^i))\rho_{\eps,r,r_0}^\Dx}_{L^1(\Pi_T^{r_0} \times \R)}
    \le \norm{\chi(A(u_\Dx);\cdot)}_{L^1([0,T] \times \R;BV(\R^d))}\\
    \times \norm{\varphi^i_1 + \varphi^i_2}_{L^1(\R^d)}.
  \end{multline*}
  Note that $\norm{\chi(A(u_\Dx);\cdot)}_{L^1([0,T] \times
    \R;BV(\R^d))} = \norm{A(u_\Dx)}_{L^1([0,T];BV(\R^d))}$. Now, as 
  $\partial_{x_i}J_r(x_i + z) \le r^{-3}
  \norm{J''}_\infty\car{\abs{x_i+ z} \le r}$, it follows that 
  \begin{align*}
    \norm{\varphi^i_1}_{L^1(\R^d)} &= \frac{1}{2\Dx}\int_\R
    \Bigl|\int_0^{\Dx}(z-\Dx)^2 \partial_{x_i}^2J_r(x_i+z)\,dz\Bigr|\, dx_i 
    \\
    &\le \frac{(r + \Dx)}{\Dx
      r^3}\norm{J''}_\infty\int_0^{\Dx}(z-\Dx)^2\,dz 
    \\
    &\le \frac{1}{3}\norm{J''}_\infty \frac{\Dx^2}{r^2}\left(1 +
      \frac{\Dx}{r}\right).
  \end{align*}
  The same estimate applies to $\varphi_2^i$.
\end{proof}

\begin{estimate} \label{est:5}
  Let $\term$ be the term \eqref{eq:T7} from 
  Lemma \ref{lemma:ContractionLemma}, that is,
  \begin{equation*}
    \term = \int_R (\signe{\zeta-A(u)} \star J_{r_0} \otimes
    J_r)(\Delta - \Dm \cdot \Dp)\rho_{\eps,r,r_0}^\Dx \,d\zeta.
  \end{equation*}
  Then there exists a constant $C = C(d,J)$ such that
  \begin{equation*}
    \norm{\term}_{L^1(\Pi_T^{r_0})} \le C\frac{\Dx^2}{r^3}\left(1 +
      \frac{\Dx}{r}\right)\norm{A(u_\Dx)}_{L^1([0,T];BV(\R^d))}. 
  \end{equation*}
\end{estimate}
\begin{remark}
  At the cost of an extra factor $r^{-1}$, the $BV$ 
  norm may be replaced by the $L^1$ norm.
\end{remark}
\begin{proof}
  First note that $\abs{\signe{\zeta-A(u)} \star J_{r_0} \otimes J_r}
  \le 1$, so
  \begin{equation*}
    \abs{\term} \le \norm{(\Delta - \Dm \cdot
      \Dp)\rho_{\eps,r,r_0}^\Dx}_{L^1(\Pi_T \times \R)}. 
  \end{equation*}
  Now,
  \begin{equation*}
    (\partial_{x_i}^2 - \Dm^i\Dp^i)\rho_{\eps,r,r_0}^\Dx 
    = \chi(A(u_\Dx);\cdot) \star J_\eps \otimes (\partial_{x_i}^2 -
    \Dm^i\Dp^i)J_r \otimes J_{r_0}. 
  \end{equation*}
  Using a Taylor expansion \cite[p.24]{KRS2014},
  \begin{align*}
    (\partial_{x_i}^2-\Dm^i\Dp^i)J_r(x)
    &= \frac{1}{6\Dx^2}\int_0^\Dx (z-\Dx)^3\partial_{x_i}^4J_r(x_i +
    z)\,dz \,\prod_{j \neq i}J_r(x_j)
    \\
    &\qquad - \frac{1}{6\Dx^2}\int_{-\Dx}^0
    (z+\Dx)^3\partial_{x_i}^4J_r(x_i + z)\,dz \,\prod_{j \neq i}J_r(x_j) 
    \\
    &=: \partial_{x_i}(\varphi_1^i(x) + \varphi_2^i(x)).
  \end{align*}
  Hence,
  \begin{equation*}
    (\partial_{x_i}^2 - \Dm^i\Dp^i)\rho_{\eps,r,r_0}^\Dx 
    =\partial_{x_i}\left(\chi(A(u_\Dx);\cdot) \star J_\eps \otimes
      (\varphi_1^i + \varphi_2^i)  \otimes J_{r_0}\right). 
  \end{equation*}
  By Young's inequality for convolutions,
  \begin{multline*}
    \norm{(\partial_{x_i}^2 -
      \Dm^i\Dp^i)\rho_{\eps,r,r_0}^\Dx}_{L^1(\Pi_T \times \R)}
    \le \norm{\chi(A(u_\Dx);\cdot)}_{L^1([0,T] \times \R;BV(\R^d))}\\
    \times \norm{\varphi^i_1 + \varphi^i_2}_{L^1(\R^d)}.
  \end{multline*}
  It remains to estimate the $L^1$ 
  norm of $\varphi^i_1$ and $\varphi^i_2$:
  \begin{align*}
    \norm{\varphi^i_1}_{L^1(\R^d)}
    &= \frac{1}{6\Dx^2}\int_\R \Bigl|\int_0^\Dx
      (z-\Dx)^3\partial_{x_i}^3J_r(x_i + z)\,dz\Bigr| \,dx_i 
    \\
    &\le \frac{r + \Dx}{3\Dx^2r^4}\norm{J^{(3)}}_\infty
    \Bigl|\int_0^\Dx (z-\Dx)^3 \,dz\Bigr|
    \\
    &=\frac{\norm{J^{(3)}}_\infty}{12}\frac{\Dx^2}{r^3}\left(1 +
      \frac{\Dx}{r}\right).
  \end{align*}
  A similar estimate applies to $\varphi^i_2$.
\end{proof}
\begin{estimate}\label{est:6}
  Let $\term_1$ and $\term_2$ be the terms from \eqref{eq:T6} 
  in Lemma~\ref{lemma:ContractionLemma}, that is,
  \begin{align*}
    \term_1 &= \int_\R \left(\signe{\zeta-A(u)} \star J_{r_0}
      \otimes J_r\right)G_1'(\zeta)\cdot
    (\Dp-\nabla)\rho_{\eps,r,r_0}^\Dx \,d\zeta, 
    \\
    \term_2 &= \int_\R \left(\signe{\zeta-A(u)} \star J_{r_0}
      \otimes J_r\right)G_2'(\zeta)\cdot
    (\Dm-\nabla)\rho_{\eps,r,r_0}^\Dx \,d\zeta,
  \end{align*}
  where $G_k(A(u))=F_k(u)$ for $j=1,2$. Then there exists 
  a constant $C = C(d,J)$ such that
  \begin{equation*}
    \norm{\term_k}_{L^1(\Pi_T^{r_0})} \le C  \frac{\Dx}{r}\left(1 +
      \frac{\Dx}{r}\right)\norm{F_k'}_{L^\infty(\R;\R^d)}
    \norm{A(u_\Dx)}_{L^1([0,T];BV(\R^d))} 
  \end{equation*}
  for  $k = 1,2$.
\end{estimate}

\begin{remark}
 Again the $BV$ norm may be replaced by the
 $L^1$ norm at the cost of an extra factor $r^{-1}$.
\end{remark}

\begin{proof}
  Consider $\term_1$. We can change variables $\zeta=A(\xi)$, which yields
  \begin{equation*}
     \term_1 = \int_\R \left(\signe{A(\xi)-A(u)} \star J_{r_0}
      \otimes J_r\right)F_1'(\zeta)\cdot
    (\Dp-\nabla)\rho_{\eps,r,r_0}^\Dx \,d\xi.
  \end{equation*}
  Then observe that
  \begin{align*}
    \norm{\term_1}_{L^1(\Pi_T^{r_0})} &\le \norm{F_1' \cdot (\Dp
      -\nabla)\left(\chi(A(u_\Dx);\cdot) \star J_\eps \otimes
        J_r \otimes J_{r_0}\right)}_{L^1(\Pi_T \times \R)}.
  \end{align*}
  We have
  \begin{multline*}
    (\Dp^i -\partial_{x_i})\left(\chi(A(u_\Dx);\cdot) \star J_\eps
      \otimes J_r \otimes J_{r_0}\right) 
    \\
    = \chi(A(u_\Dx);\cdot) \star J_\eps \otimes (\Dp^i
    -\partial_{x_i})J_r \otimes J_{r_0}.
  \end{multline*}
  By Taylor expansions,
  \begin{align*}
    (\Dp^i -\partial_{x_i})J_r(x) &= \frac{1}{\Dx}\int_0^\Dx
    (\Dx-z)\partial_{x_i}^2 J_r(x_i+z)\,dz \prod_{j \neq i} J_r(x_j) 
    \\
    &=: \partial_{x_i}\varphi(x).
  \end{align*}
  By Young's inequality for convolutions,
  \begin{multline*}
    \norm{(F_1^i)'(\Dp^i -\partial_{x_i})\left(\chi(A(u_\Dx);\cdot)
        \star J_\eps \otimes J_r \otimes J_{r_0}\right)}_{L^1(\Pi_T
      \times \R)} 
    \\
    \le \norm{(F_1^i)'}_{L^\infty(\R)}
    \norm{\chi(A(u_\Dx);\cdot)}_{L^1([0,T] \times
      \R;BV(\R^d))}\norm{\varphi}_{L^1(\R^d)}.
  \end{multline*}
  It remains to estimate $\norm{\varphi}_{L^1(\R^d)}$:
  \begin{align*}
    \norm{\varphi}_{L^1(\R^d)} &= \frac{1}{\Dx}\int_\R \Bigl|\int_0^\Dx
      (\Dx-z)\partial_{x_i} J_r(x_i+z)\,dz\Bigr|\,dx_i 
    \\
    &\le \frac{1}{\Dx}2\norm{J'}_\infty\frac{r + \Dx}{r^2} \int_0^\Dx
    (\Dx-z)\,dz 
    \\
    &= \norm{J'}_\infty \frac{\Dx}{r}\left(1 + \frac{\Dx}{r}\right),
  \end{align*}
  from which the estimate of $\term_1$ follows. 
  Similar arguments apply to $\term_2$.
\end{proof}

\begin{estimate} \label{est:7}
  Consider the terms \eqref{eq:T2}, \eqref{eq:T3}, \eqref{eq:T4}, 
  and \eqref{eq:T5} from Lemma \ref{lemma:ContractionLemma}. 
  Suppose $A' > \eta$ and set $B:= A^{-1}$. Let
  \begin{equation*}
    \begin{split}
      \term_1 &= \int_\R \left(\signe{\zeta-A(u_\Dx)}\star J_{r_0} \otimes
        J_r\right)\partial_tR_{\eps,r,r_0}^{B'}(\zeta) \,d\zeta, 
      \\
      \term_2 &= \int_\R \left(\signe{\zeta-A(u)}\star J_{r_0} \otimes
        J_r\right)\partial_tR_{\eps,r,r_0}^{B',\Dx}(\zeta) \,d\zeta, 
      \\
      \term_3 &= \int_\R \left(\signe{\zeta-A(u_\Dx)}\star J_{r_0} \otimes
        J_r\right)\nabla \cdot R_{\eps,r,r_0}^{g'}(\zeta) \,d\zeta, 
      \\
      \term_4 &= \int_\R \left(\signe{\zeta-A(u)}\star J_{r_0} \otimes
        J_r\right) 
      \\
      &\hphantom{XXXXXXXXXXXXX} \times \left(\Dp \cdot
        R_{\eps,r,r_0}^{G_1',\Dx}(\zeta) + \Dm \cdot
        R_{\eps,r,r_0}^{G_2',\Dx}(\zeta)\right) \,d\zeta,
    \end{split}
  \end{equation*}
  where
  \begin{equation*}
    R_{\eps,r,r_0}^{f}(\zeta) = R_\eps^{f}(A(u),\zeta) \star J_{r_0} \otimes J_r
    \text{ and }
    R_{\eps,r,r_0}^{f,\Dx}(\zeta) = R_\eps^{f}(A(u_\Dx),\zeta) \star
    J_{r_0} \otimes J_r 
  \end{equation*}
  for any function $f$, and $R_\eps^{f}$ is defined in
  equation \eqref{eq:RfDefinition}. Then
  \begin{align*}
    \norm{\term_k}_{L^\infty(\Pi_T^{r_0})} &\le 8\frac{\eps}{\eta
      r_0}\norm{J'}_\infty \mbox{ for $k = 1,2$}, 
    \\
    \norm{\term_3}_{L^\infty(\Pi_T^{r_0})} &\le 8\frac{\eps}{\eta
      r}\norm{J'}_\infty \sum_{i=1}^d \norm{f_i}_\Lip, 
    \\
    \norm{\term_4}_{L^\infty(\Pi_T^{r_0})} &\le 8\frac{\eps}{\eta
      r}\left(1 + \frac{\Dx}{r}\right)\norm{J'}_\infty
    \sum_{k=1}^2\sum_{i=1}^d \norm{F_k^i}_\Lip.
  \end{align*}
\end{estimate}

\begin{proof}
  Consider $\term_1$. Moving the $t$ derivative onto $J_{r_0}$, we have that
  \begin{equation*}
    \term_1 = \int_\R \signe{\zeta-A(u_\Dx)} R_{\eps}^{B'}(A(u),\zeta)
    \,d\zeta \conv{u_\Dx,u} J_{r_0} \otimes J_r
    \otimes \partial_tJ_{r_0} \otimes J_r. 
  \end{equation*}
  By Lemma~\ref{lemma:doubleConvolutionEstimate}, 
  equation~\eqref{eq:lemma:doubleConvolutionEstimateInftyBound}, 
  Lemma~\ref{lemma:MollifiedChiProd}, and 
  equation~\eqref{eq:IntByPartsBoundOnRGDer} with $f(z) = z$,
  \begin{equation*}
    \norm{\term_1}_{L^\infty(\Pi_T^{r_0})} \le 4\frac{\eps}{\eta}\norm{J_{r_0}
      \otimes J_r}_{L^1(\R \times \R^d)}\norm{\partial_tJ_{r_0} \otimes
      J_r}_{L^1(\R \times \R^d)} \le 8\frac{\eps}{\eta r_0}\norm{J'}_\infty. 
  \end{equation*}
  The $L^\infty$ bound on $\term_2$ follows similarly. 
  
  Let us consider $\term_3$:
  \begin{equation*}
    \term_3 = \sum_{i=1}^d\int_\R
    \signe{\zeta-A(u_\Dx)}R_{\eps}^{g_i'}(A(u),\zeta) \,d\zeta
    \conv{u_\Dx,u} J_{r_0} \otimes J_r \otimes J_{r_0}
    \otimes \partial_{x_i}J_r. 
  \end{equation*}
  By Lemma~\ref{lemma:doubleConvolutionEstimate}, 
  equation~\eqref{eq:lemma:doubleConvolutionEstimateInftyBound}, 
  Lemma~\ref{lemma:MollifiedChiProd}, and 
  equation~\eqref{eq:IntByPartsBoundOnRGDer} with $f(z) = f_i(z)$,
  \begin{equation*}
    \norm{\term_3}_{L^\infty(\Pi_T^{r_0})} \le 4\frac{\eps}{\eta}\sum_{i=1}^d
    \norm{f_i}_\Lip\norm{\partial_{x_i}J_r}_{L^1(\R^d)} \le
    8\frac{\eps}{\eta r}\norm{J'}_\infty \sum_{i=1}^d
    \norm{f_i}_\Lip. 
  \end{equation*}

  The terms in $\term_4$ are estimated in the
  same way, but in view \eqref{eq:FiniteDiffofMollifierEst} we 
  can utilize the bound 
  \begin{equation*}
    \norm{J_{r_0} \otimes D_{\pm}J_r}_{L^1(\R \times \R^d)} \le
    2\norm{J'}_\infty \frac{1}{r}\left(1 + \frac{\Dx}{r}\right). 
  \end{equation*}
\end{proof}

\subsection{Concluding the proof of 
Theorem \ref{theorem:MultidLocal}}\label{subsec:mainproof}
Recall that $Q_\varepsilon$, cf.~\eqref{eq:Lepsdef}, was introduced 
as an approximation to the contraction functional $Q$, cf.~\eqref{eq:QDef}.
Recall the basic property \cite{ChenPerthame2003,Perthame1998}
\begin{equation} \label{eq:AbsQExpression}
  \abs{u-v} = \int_\R Q(u,v;\xi)\,d\xi, \qquad u,v\in \R.
\end{equation}
To argue for this relation, note that 
  \begin{equation*}
    Q(u,v;\xi) = \abs{\chi(u;\xi)} +
    \abs{\chi(v;\xi)}-2\chi(u;\xi)\chi(v;\xi) = (\chi(u;\xi)-\chi(v;\xi))^2.
  \end{equation*}
 Next, observe that 
  \begin{equation}\label{eq:ChiDiffEqu}
    \chi(u;\xi)-\chi(v;\xi) = \chi(u-v;\xi - v);
  \end{equation}
  indeed, for any $S \in C^1_0(\R)$, 
  \begin{align*}
    \int_\R S'(\xi)(\chi(u;\xi)-\chi(v;\xi))\,d\xi 
    &= \int_u^v S'(\xi)\,d\xi
    = \int_0^{u-v}S'(\sigma + v)\,d\sigma 
    \quad \text{(here $\sigma = \xi-v$)}\\
    &= \int_\R S'(\sigma + v)\chi(u-v;\sigma)\,d\sigma \\
    &= \int_\R S'(\xi)\chi(u-v;\xi-v)\,d\xi.
  \end{align*}
  Hence, the claim follows: 
  \begin{equation*}
    \int_\R (\chi(u;\xi)-\chi(v;\xi))^2 \,d\xi = \int_\R
    \abs{\chi(u-v;\xi-v)}\,d\xi = \abs{u-v}. 
  \end{equation*}

Let us quantify the approximation properties of $Q_\eps$.
\begin{lemma}
  \label{lem:absdiffs}
  Let $A' \ge \eta>0$, $B = A^{-1}$, and $f = g \circ A$. Define
  \begin{align*}
    P &= \int_\R Q_\eps(A(u),A(v);\zeta)B'(\zeta)\,d\zeta-\abs{u-v},\\
    M&= \int_\R Q_\eps(A(u),A(v);\zeta) g'(\zeta)\,d\zeta - \sign{u-v}
    \left(f(u)-f(v)\right),
   \intertext{and}
    N&=\int_\R Q_\eps(A(u),A(v);\zeta)\,d\zeta - \abs{A(u)-A(v)},
  \end{align*}
  for any $u$ and $v$, and where 
  $Q_{\eps}$ is given by \eqref{eq:Lepsdef}. 
  Then 
  \begin{equation*}
    \abs{P} \le 16 \frac{\eps}{\eta},\qquad
    \abs{M}\le 8\frac{\eps}{\eta}, \qquad
    \abs{N}\le 8\frac{\eps}{\eta}.
  \end{equation*}
\end{lemma}
\begin{proof}
  Because $A'>0$, $Q(u,v;\xi) =
  Q(A(u),A(v);A(\xi))$. Hence we can use
  \eqref{eq:AbsQExpression} and a change of 
  variables to obtain the identify 
  \begin{equation*}
    P = \int_\R
    \left(Q_\eps(A(u),A(v);\zeta)-Q(A(u),A(v);\zeta)\right)B'(\zeta)\,d\zeta. 
  \end{equation*}
  By definition of $Q$ and the equality
  $\abs{\chi(u;\xi)} = \sign{\xi}\chi(u;\xi)$,
  \begin{align*}
    	Q(A(u),A(v);\zeta) & = \sign{\zeta}\chi(A(u);\zeta) +
    	\sign{\zeta}\chi(A(v);\zeta) 
	\\ & \qquad\qquad 
	-2\chi(A(u);\zeta)\chi(A(v);\zeta). 
  \end{align*}
  Thus,
  \begin{align*}
    P&= \int_\R \left(\signe{\zeta}\chi_\eps(A(u);\zeta)
      -\sign{\zeta}\chi(A(u);\zeta)\right)B'(\zeta) \,d\zeta 
    \\
    &\quad + \int_\R \left(\signe{\zeta}\chi_\eps(A(v);\zeta)-
      \sign{\zeta}\chi(A(v);\zeta)\right)B'(\zeta) \,d\zeta 
    \\
    &\quad + 2\int_\R
    \left(\chi(A(u);\zeta)\chi(A(v);\zeta)-\chi_\eps(A(u);\zeta)
      \chi_\eps(A(v);\zeta))\right)B'(\zeta)\,d\zeta  
    \\
    &=: P_1 + P_2 + P_3.
  \end{align*}
  Finding that the measure of the support of the integrand is bounded
  by $4\eps$ for $P_1$, $P_2$, and $P_3$, we conclude that
  \begin{equation*}
    \abs{P} \le 16 \eps \norm{B'}_\infty,
  \end{equation*}
  and then the bound on $P$ follows since $\norm{B'}_\infty \le \eta^{-1}$. 
  
  To prove the inequality for $M$, note that
  \begin{align*}
    \sign{u-v}&(f(u)-f(v))\\&=\int_\R
    \sign{u-v}(\chi(u;\zeta)-\chi(v;\zeta))f'(\zeta)\,d\zeta\\
    &=\int_\R \abs{\chi(u,\zeta)-\chi(v;\zeta)}f'(\zeta)\,d\zeta\\
    &=\int_\R
    \left[\sign{\zeta}\chi(u;\zeta)+\sign{\zeta}\chi(v;\zeta)
    -2\chi(u;\zeta) \chi(v;\zeta)\right]f'(\zeta)\,d\zeta.
  \end{align*}
  Changing variables, we arrive at 
  \begin{equation*}
    \int_\R Q_\eps(A(u),A(v);\zeta)g'(\zeta) \,d\zeta= 
    \int_\R Q_\eps(A(u),A(v);A(\zeta)) f'(\zeta)\,d\zeta,
  \end{equation*}
  and, since $\sign{\zeta}\chi(w;\zeta)
  =\sign{A(\zeta)}\chi(A(w);A(\zeta))$, we find that
  \begin{align*}
    \abs{M}&\le \int_\R \Bigl| \signe{A(\zeta)}\chi_\eps(A(u);A(\zeta))
      - \sign{A(\zeta)}\chi(A(u);A(\zeta))\Bigr| \,\abs{f'(\zeta)}\,d\zeta
    \\
    &\, + \int_\R \Bigl|\signe{A(\zeta)}\chi_\eps(A(v);A(\zeta)) -
      \sign{A(\zeta)}\chi(A(v);A(\zeta))\Bigr| \,\abs{f'(\zeta)}\,d\zeta
    \\
    &\, + 2\!\! \int_\R
    \abs{\chi_\eps(A(u);A(\zeta))\chi_\eps(A(v);A(\zeta)) -
      \chi(A(u);A(\zeta))\chi(A(v);A(\zeta))} \abs{f'(\zeta)}\,d\zeta.
  \end{align*}
  Each of the three integrands is bounded by $2$ and has support
  where $\abs{A(\zeta)}<\eps$, i.e., where $\abs{\zeta}\le \eps/\eta$,
  hence $\abs{M}\le 8\eps/\eta$. The proof of the bound on $\abs{N}$ is similar.
\end{proof}

\begin{proof}[Concluding the proof of
  Theorem~\ref{theorem:MultidLocal}]
  We shall choose a positive test function $\phi\le 1$, such that $\abs{\nabla
    \phi}$ and $\abs{\Delta \phi}$ are bounded by $C\phi$. This
  will be convenient when we estimate terms containing $\nabla \phi$
  or $\Delta \phi$. 

  A test function with the necessary properties can be defined as follows,
  fix $R > d\sqrt{d}$ and define $\hat{\phi}:\R^d \rightarrow \R$ by
  \begin{equation*}
    \hat{\phi}(x) = 
    \begin{cases}
      1 &\text{if $\abs{x} \le R + \sqrt{d}$,} \\
      \exp((R + \sqrt{d}-\abs{x})/\sqrt{d}) &\text{otherwise.}
    \end{cases}
  \end{equation*}
  Define $\test = \hat{\phi} \star J^{\otimes n}$ and note that
  $\test(x) = 1$ for $x \in B(0,R)$. Note that $\hat{\phi}$ is 
  weakly differentiable and satisfies
  \begin{equation*}
    \partial_{x_i}\hat{\phi}=
    \begin{cases} -\frac{x_i}{\sqrt{d}\abs{x}}\hat{\phi}(x) 
      &\abs{x}> R+\sqrt{d},\\
      0 &\abs{x}<R+\sqrt{d}.
    \end{cases} 
  \end{equation*}
  It follows that  $|\nabla \phi(x)| \le \frac{1}{\sqrt{d}}\phi(x)$.
  In order to bound $\Delta \phi$ we first note that
  \begin{equation*}
    \Delta \hat{\phi}(x) =
    \left(\frac{1}{d}-\frac{d-1}{\sqrt{d}\abs{x}}\right)\hat{\phi}(x),  
    \ \text{ for $\abs{x} > R + \sqrt{d}$.}
  \end{equation*}
  Furthermore,
  \begin{equation*}
   \frac{1}{d^2} \le \left(\frac{1}{d}-\frac{d-1}{\sqrt{d}\abs{x}}\right) 
   \le \frac{1}{d} \ \text{ for  $\abs{x} \ge d\sqrt{d}$.}
 \end{equation*}
 It follows that $|\Delta \hat{\phi}| \le \frac{1}{d}\hat{\phi}(x)$ 
  whenever $\abs{x} > R + \sqrt{d}$. Hence
  \begin{equation*}
   \abs{\Delta \phi(x)} \le \frac{1}{d} \phi(x)\ \text{ for
     $\abs{x} > R + 2\sqrt{d}$}.  
  \end{equation*}
  If $\abs{x} \le R + 2\sqrt{d}$ it follows by the 
  lower bound $\phi(x) \ge e^{-2}$, that there exists a 
  constant $C = C(d,J)$ such that $\abs{\Delta \phi(x)} \le C\abs{\phi(x)}$.

  The next lemma estimates how far $\abs{u_\Dx-u}$ is from it
  regularized counterpart
  \begin{equation*}
    \int B'(\zeta)\left(\chi_\eps(A(u_\Dx);\zeta)-\chi_\eps(A(u);\zeta)\right)^2\,d\zeta
    \conv{u_\Dx,u} J_{r_0}\otimes J_r \otimes J_{r_0}\otimes J_r. 
  \end{equation*}
\begin{lemma}\label{lem:rr0terms}
With the notation and assumptions of Lemma~\ref{lemma:ContractionLemma},
\begin{equation}
 \begin{aligned}
\int_{\R^d} \Bigl| \int_\R
B'(\zeta)\Qeps(A(u_\Dx),A(u);&\zeta)\,d\zeta -
\abs{u_\Dx-u}\Bigr| \phi \,dx \\
& \le C\left(r+r_0 +
  \norm{\phi}_{L^1(\R^d)}\frac{\eps}{\eta}\right),
 \end{aligned}
 \label{eq:rr01}
\end{equation}
\begin{equation}
 \label{eq:rr02}
 \begin{aligned}
\int_{r_0}^\tau& \int_{\R^d} \Bigl| \Bigl[\int_\R
g'(\zeta)\Qeps(A(u_\Dx),A(u);\zeta)\,d\zeta \\
&- \sign{u_\Dx-u}(f(u)-f(u_\Dx))\Bigr]\cdot\nabla\phi
\Bigr|\,dxdt \le CT\left(r+r_0 +
  \norm{\phi}_{L^1(\R^d)}\frac{\eps}{\eta}\right),
 \end{aligned}
\end{equation}
and
\begin{equation}
 \label{eq:rr03}
 \begin{aligned}
\int_{r_0}^\tau \int_{\R^d} \Bigl| \Bigl[\int_\R &
\Qeps(A(u_\Dx),A(u);\zeta)\,d\zeta \\
&- \abs{A(u_\Dx)-A(u)}\Bigr]\Delta \phi \Bigr|\,dxdt \le
CT\left(r+r_0 +
  \norm{\phi}_{L^1(\R^d)}\frac{\eps}{\eta}\right),
 \end{aligned}
\end{equation}
where the constant $C$ only depends on the initial data, $A$, and $f$.
\end{lemma}

\begin{proof}[Proof of Lemma~\ref{lem:rr0terms}]
We establish \eqref{eq:rr01} as follows:
\begin{align*}
 \int_{\R^d} \Bigl|& \int_\R
 B'(\zeta)\Qeps(A(u_\Dx),A(u);\zeta)\,d\zeta -
 \abs{u_\Dx-u}\Bigr|\phi\,dx \\
 &= \int_{\R^d} \Bigl| \int_0^T\int_{\R^d} \Bigl( \int_\R
 B'(\zeta)Q_\eps(A(u_\Dx(s,y)),A(u(s,y));\zeta)\,d\zeta\\
 &\hphantom{= \int_{\R^d} \Bigl| \int_0^T\int_{\R^d}}\quad -
 \abs{u_\Dx(t,x)-u(t,x)} \Bigr) J_{r_0}(t-s)J_r(x-y)\,dyds \Bigr| \phi \,dx\\ 
 &\le \int_{\R^d} \int_0^T\int_{\R^d} \Bigl| \int_\R
 B'(\zeta)Q_\eps(A(u_\Dx(s,y)),A(u(s,y));\zeta)\,d\zeta\\
 &\hphantom{= \int_{\R^d} \Bigl| \int_0^T\int_{\R^d}}\quad -
 \abs{u_\Dx(s,y)-u(s,y)} \Bigr| J_{r_0}(t-s)J_r(x-y)\,dyds \phi\,dx\\
 &\quad + \int_{\R^d}\int_0^T 
 \int_{\R^d}\Bigl[ \abs{u_\Dx(t,x)-u_\Dx(s,y)}+\abs{u(t,x)-u(s,y)}\Bigr]\\
 &\hphantom{\quad + \int_{\R^d}\int_0^T
\int_{\R^d}\Bigl[}\qquad\qquad\qquad\qquad \times
 J_{r_0}(t-s)J_r(x-y)\,dydsdx\\
 &\le 16\frac{\eps}{\eta} \int_{\R^d}\phi\,dx
 +2\left(\abs{u_0}_{BV(\R^d)} +
\abs{u_\Dx(0,\cdot)}_{BV(\R^d)}\right) (r+r_0).
\end{align*}
The bounds \eqref{eq:rr02} and \eqref{eq:rr03} are proved in the
same way.
 \end{proof}

Writing the equation in Lemma~\ref{lemma:ContractionLemma} as
\begin{equation*}
\int_\R B'(\zeta)\partial_t \Qeps\,d\zeta + \int_\R
g'(\zeta)\nabla \Qeps \,d\zeta = \int_\R \Delta \Qeps\,d\zeta + \Eeps,
\end{equation*}
we multiply by $\phi$, integrate over $t\in [r_0,\tau]$ 
where $r_0 < \tau \le T-r_0$, and integrate by parts
in $x$, finally obtaining
\begin{multline*}
\int_{\R^d} \int_\R B'(\zeta) \Qeps \,d\zeta \phi
\Bigm|_{t=r_0}^{t=\tau} \,dx - \int_0^\tau \int_{\R^d} \int_\R
g'(\zeta)\Qeps \cdot \nabla \phi \,d\zeta dxdt
\\ = \int_{r_0}^\tau \int_{\R^d} \int_\R \Qeps \Delta \phi \,d\zeta dxdt  
+\int_{r_0}^\tau \int_{\R^d} \Eeps \phi \,dxdt.
\end{multline*}
Combining this with Lemma~\ref{lem:rr0terms} gives
\begin{multline*}
\int_{\R^d} \abs{u_\Dx - u}\phi\Bigm|_{t=r_0}^{t=\tau} \,dx
- \int_{r_0}^\tau \int_{\R^d} \sign{u_\Dx-u}(f(u_\Dx)-f(u))
\cdot\nabla \phi \,dxdt \\
\le \int_{r_0}^\tau \int_{\R^d} \abs{A(u_\Dx)-A(u)}\Delta \phi\,dxdt +
\int_{r_0}^\tau \int_{\R^d} \Eeps \phi \,dxdt\\
+ C_T\left(r+r_0+\frac{\eps}{\eta}\right),
\end{multline*}
where $C_T$ depends (linearly) on $T$. Using properties of
$\phi$, this can be rewritten as
\begin{equation*}
\Lambda(\tau)-\Lambda(r_0) \le C \int_{r_0}^\tau \Lambda(t)\,dt + \BEeps,
\end{equation*}
where
\begin{align*}
\Lambda(t)&=\int_{\R^d} \abs{u_\Dx(t,x)-u(t,x)}\phi(x)\,dx,\\
\BEeps &= \int_{r_0}^\tau \int_{\R^d} \Eeps \phi \,dxdt +
C_T\left(r+r_0+\frac{\eps}{\eta}\right).
\end{align*}
Gronwall's inequality then implies that
\begin{equation*}
\Lambda(\tau)\le \Lambda(r_0)
+\tau e^{C\tau}\left(\Lambda(r_0)+\BEeps\right).
\end{equation*}

Recall that $u$ depends on $\eta$, and we now make this dependence
explicit by writing $u^\eta$ and $\Lambda^\eta$. Our aim is to
estimate $u_\Dx - u^0$. By \eqref{eq:uminusueta},
\begin{align*}
& \int_{B(0,R)} \abs{u_\Dx(\tau,\cdot)-u^0(\tau,\cdot)}\,dxdt -
C\sqrt{\eta} \\
&\qquad \qquad \le \Lambda^\eta(\tau)\\
&\qquad \qquad \le C_T
\norm{u_\Dx(r_0,\cdot)-u^0(r_0,\cdot)}_{L^1(\R^d)} + C_T \BEeps\\
&\qquad \qquad \le 
Cr_0 + \norm{u_\Dx(0,\cdot)-u_0}_{L^1(\R^d)} +C_T \BEeps.
\end{align*}

Next, we estimate the terms in the integral of $\Eeps$; these are
the terms in \eqref{eq:T1}--\eqref{eq:T9}. 
By Estimate~\ref{est:4},
\begin{equation}
\label{eq:est4bnd}
\iint_{\Pi_T^{r_0}} \abs{\text{second term in \eqref{eq:T1}}}\,dxdt\le
C\frac{\Dx^2}{r^3}\left(1+\frac{\Dx}{r}\right),
\end{equation}
where $C$ depends on the initial data. 

The integral of the terms
\eqref{eq:T2}--\eqref{eq:T5} is bounded by Estimate~\ref{est:7}:
\begin{equation*}
\iint_{\Pi_T^{r_0}} \abs{\text{\eqref{eq:T2}} + \ldots + \text{\eqref{eq:T5}}}\,dxdt \le
C\frac{\eps}{\eta}\left(\frac{1}{r_0}+\frac{1}{r}\left(1+\frac{\Dx}{r}\right)\right).
\end{equation*}

The integral of \eqref{eq:T6} is bounded by
Estimate~\ref{est:6} as follows:
\begin{equation*}
\iint_{\Pi_T^{r_0}}\abs{\text{\eqref{eq:T6}}}\,dxdt 
\le C\frac{\Dx}{ r}\left(1+\frac{\Dx}{r}\right).
\end{equation*}

The integral of \eqref{eq:T7} is bounded using Estimate~\ref{est:5}:
\begin{equation*}
\iint_{\Pi_T^{r_0}} \abs{\text{\eqref{eq:T7}}}\,dxdt 
\le C\frac{\Dx^2}{r^3}\left(1+\frac{\Dx}{r}\right). 
\end{equation*}

The term \eqref{eq:T9} is bounded using 
Estimates \ref{est:1} and \ref{est:2} (if $d>1$):
\begin{subequations}
\begin{equation}\label{eq:est9bnda}
 \iint_{\Pi_T^{r_0}}\abs{\text{\eqref{eq:T9}}}\,dxdt \le 
 C\Dx\left(\frac{1}{r^2}+\frac{1}{r} 
 + \frac{1}{\eps^2\sqrt{r_0r^d}}\right).
\end{equation}
If $d=1$, we can use Estimate~\ref{est:3} 
to achieve the better bound
\begin{equation}
 \label{eq:est9bndb}
 \iint_{\Pi_T^{r_0}}\abs{\lim_{\eps\to 0}\text{\eqref{eq:T9}}}\,dxdt \le 
 C\Dx\left(\frac{1}{r^2}+\frac{1}{r}+\frac{1}{r_0}\right).
\end{equation}
\end{subequations}
Finally, the term \eqref{eq:T8} is non-positive.

The fraction $\Dx/r$ will turn out to be uniformly bounded (in fact
vanishingly small), so we can overestimate it by a constant.  Thus
the bounds \eqref{eq:est4bnd}--\eqref{eq:est9bndb} 
give the following estimate for $\BEeps$:
\begin{equation*}
\BEeps \le C_T \left(
 r+r_0+\frac{\eps}{\eta} + \frac{\eps}{\eta r_0} + \frac{\eps}{\eta
r} + \frac{\Dx}{r} + \frac{\Dx}{r^2} +
 \frac{\Dx}{\eps^2\sqrt{r_0 r^d}}\right).
\end{equation*}
If $u_0\in BV(\R^d)$, $\norm{u_\Dx(0,\cdot)-u_0}_{L^1(\R^d)} \le
\abs{u_0}_{BV(\R^d)}\Dx$, so that
\begin{multline*}
\norm{u_{\Dx}(\tau,\cdot)-u^0(\tau,\cdot)}_{L^1(B(0,R))} \\
\le C_T \left(\Dx + \sqrt{\eta} + r + r_0 +
 \frac{\eps}{\eta}+\frac{\eps}{\eta r_0} + \frac{\eps}{\eta r} +
 \frac{\Dx}{r} + \frac{\Dx}{r^2} + \frac{\Dx}{\eps^2 \sqrt{r_0
  r^d}}\right).
\end{multline*}
Now we set $r=r_0=\sqrt{\eta}$, $\eps=r^4$; using that $r<1$, the
above simplifies to
\begin{equation*}
\norm{u_{\Dx}(\tau,\cdot)-u^0(\tau,\cdot)}_{L^1(B(0,R))} \le C_T 
\left(r+ \frac{\Dx}{r^{\tfrac{17+d}{2}}}\right).
\end{equation*}
Finally, minimizing with respect to $r$ yields
\begin{equation*}
\norm{u_{\Dx}(\tau,\cdot)-u^0(\tau,\cdot)}_{L^1(B(0,R))} \le C_T \Dx^{\tfrac{2}{19+d}}.
\end{equation*}
\end{proof}

\begin{remark}
If $d=1$, the above estimate gives a convergence rate of $1/10$,
which is better than the rate reported in \cite{KKR2012}. 
However, when $d=1$, we can use \eqref{eq:est9bndb} instead of
\eqref{eq:est9bnda}. Then we have no terms with $\eps$ in the
denominator, so we can send $\eps$ to zero in \eqref{eq:T0}--\eqref{eq:T9} 
before taking absolute values and integrating.  Proceeding as above, i.e., setting 
$r=r_0=\sqrt{\eta}$, this yields the bound
\begin{equation*}
	\norm{u_{\Dx}(\tau,\cdot)-u^0(\tau,\cdot)}_{L^1(B(0,R))} 
	\le C_T\left(r+\frac{\Dx}{r^2}\right),
\end{equation*}
which gives the rate $1/3$ \cite{KRS2014}.
\end{remark}

\appendix

\section{Well-posedness of difference method}
\label{app:semiexist}
In this appendix we establish the well-posedness of the 
semi-discrete method. We also collect a series of a priori bounds. 

Introduce
\begin{equation*}
  \|\sigma\|_1 = \Dx^d \sum_\alpha |\sigma_\alpha| \quad 
  \text{and} \quad |\sigma|_{BV} = \sum_\alpha 
  \sum_{i = 1}^d |\sigma_{\alpha + e_i}-\sigma_\alpha|.
\end{equation*}
If these quantities are bounded we say 
that $\sigma = \{\sigma_\alpha\}$ is in $\ell^1(\Z^d)$ and of 
bounded variation. Let $u(t) = \{u_\alpha(t)\}_{\alpha
  \in \Z^d}$ and $u_0 = \{u_\alpha(0)\}_{\alpha \in \Z^d}$ and define
the operator $\mathcal{A}:\ell^1(\Z^d) \rightarrow \ell^1(\Z^d)$ by
\begin{equation*}
  (\mathcal{A}(u))_\alpha = \sum_{i = 1}^d 
  \Dm^i\left[F^i(u_\alpha,u_{\alpha +e_i})-\Dp^iA(u_\alpha)\right].
\end{equation*}
Then \eqref{eq:Semi-discreteScheme} may be considered as the
Cauchy problem
\begin{equation*}
  \begin{cases}
    \frac{du}{dt} + \mathcal{A}(u) = 0, & t > 0,\\
    u(0) = u_0.
  \end{cases}
\end{equation*}
This problem has a unique continuously differentiable solution for
small $t$, since
$\mathcal{A}$ is Lipschitz continuous for each $\Dx > 0$. The solution
defines a strongly continuous semigroup $\mathcal{S}(t)$ on
$\ell^1$. We want to show that this semigroup is $\ell^1$
contractive. This follows by the theory presented in
\cite{CrandallLiggett1971}, given that $\mathcal{A}$ is
accretive, i.e., 
\begin{equation*}
  \sum_\alpha \sign{u_\alpha-v_\alpha}(\mathcal{A}(u)-\mathcal{A}(v))_\alpha \ge 0.
\end{equation*}
for any $u$ and $v$ in $\ell^1(\Z^d)$ \cite{Pavel1984, Sato1968}.
\begin{lemma}
 The operator $\mathcal{A}:\ell^1(\Z^d) \rightarrow \ell^1(\Z^d)$ is
 accretive.
\end{lemma}
\begin{proof} By definition
 \begin{equation*}
(\mathcal{A}(u)-\mathcal{A}(v))_\alpha 
= \sum_{i = 1}^d \Dm^i\left[F^i(u_\alpha,u_{\alpha +e_i})
-F^i(v_\alpha,v_{\alpha + e_i})-\Dp^i(A(u_\alpha)-A(v_\alpha))\right].
 \end{equation*}
 Let $\partial_1F^i$ and $\partial_2F^i$ denote the partial
 derivatives of $F^i$ with respect to the first and second variable
 respectively. Since $F^i$ is continuously differentiable there exist
 for each $(\alpha,i)$ some number $\tau_{\alpha,i}$ such that
 \begin{equation*}
F^i(u_\alpha,u_{\alpha + e_i}) - F^i(v_\alpha,u_{\alpha+ e_i}) 
= \partial_1F^i(\tau_{\alpha,i},u_{\alpha + e_i})(u_\alpha-v_\alpha)
 \end{equation*}
 and similarly a number $\theta_{\alpha,i}$ such that
 \begin{equation*}
F^i(v_\alpha,u_{\alpha + e_i}) - F^i(v_\alpha,v_{\alpha+ e_i}) 
= \partial_2F^i(v_\alpha,\theta_{\alpha,i})(u_{\alpha + e_i}-v_{\alpha + e_i}).
 \end{equation*}
 Let $w_\alpha = u_\alpha-v_\alpha$ then
 \begin{multline*}
F^i(u_\alpha,u_{\alpha +e_i})-F^i(v_\alpha,v_{\alpha + e_i}) \\
=  F^i(u_\alpha,u_{\alpha +e_i})- F^i(v_\alpha,u_{\alpha+ e_i}) 
+ F^i(v_\alpha,u_{\alpha+ e_i})-F^i(v_\alpha,v_{\alpha + e_i}) \\
= \partial_1F^i(\tau_{\alpha,i},u_{\alpha + e_i})w_\alpha
+ \partial_2F^i(v_\alpha,\theta_{\alpha,i})w_{\alpha + e_i}.
 \end{multline*}
 Let $A' = a$. Then there exist some $\xi_\alpha$ such that
 \begin{equation*}
A(u_\alpha)-A(v_\alpha) = a(\xi_\alpha)w_\alpha.
 \end{equation*}
 Using these expressions we obtain
 \begin{equation}\label{eq:AccretiveComp}
\begin{split}
  \sum_\alpha &\sign{u_\alpha-v_\alpha}(\mathcal{A}(u)-\mathcal{A}(v))_\alpha \\
  &= \sum_\alpha\sum_{i = 1}^d \sign{w_\alpha}\Dm^i\left[\partial_1 
  F^i(\tau_{\alpha,i},u_{\alpha + e_i})w_\alpha 
  + \partial_2F^i(v_\alpha,\theta_{\alpha,i})w_{\alpha + e_i}\right] \\
  & \qquad 
  -\sum_\alpha\sum_{i = 1}^d
  \sign{w_\alpha}\Dm^i\Dp^i(a(\xi_\alpha)w_\alpha) := \term_1 - \term_2.
\end{split}
 \end{equation}
 Consider $\term_1$ first. Since
 \begin{multline*}
\Dm^i\left[\partial_1F^i(\tau_{\alpha,i},u_{\alpha + e_i})w_\alpha
  + \partial_2F^i(v_\alpha,\theta_{\alpha,i})w_{\alpha +
    e_i}\right]
= \frac{1}{\Dx}\Bigl[\partial_1F^i(\tau_{\alpha,i},u_{\alpha + e_i})w_\alpha \\
-\partial_1F^i(\tau_{\alpha-e_i,i},u_\alpha)w_{\alpha-e_i}
+ \partial_2F^i(v_\alpha,\theta_{\alpha,i})w_{\alpha +
  e_i}-\partial_2F^i(v_{\alpha-e_i},\theta_{\alpha-e_i,i})w_{\alpha}\Bigr],
 \end{multline*}
 it follows that
 \begin{equation*}
\begin{aligned}
  \term_1 &=  \frac{1}{\Dx}\sum_\alpha\sum_{i = 1}^d
  \bigg[\partial_1F^i(\tau_{\alpha,i},u_{\alpha + e_i})|w_\alpha|
  -  \partial_1F^i(\tau_{\alpha-e_i,i},u_\alpha)\sign{w_\alpha}w_{\alpha-e_i} \\
  & \qquad
  + \partial_2F^i(v_\alpha,\theta_{\alpha,i})\sign{w_\alpha}w_{\alpha
    +
    e_i}-\partial_2F^i(v_{\alpha-e_i},\theta_{\alpha-e_i,i})|w_{\alpha}|\bigg]
  \\ 
  &= \frac{1}{\Dx}\sum_{i = 1}^d
  \Bigl[\sum_\alpha \partial_1F^i(\tau_{\alpha,i},u_{\alpha +
    e_i})|w_\alpha|
  -\sum_\alpha \partial_1F^i(\tau_{\alpha,i},u_{\alpha +
    e_i})\sign{w_{\alpha + e_i}}w_{\alpha} \\ 
  &\qquad  +
  \sum_\alpha \partial_2F^i(v_\alpha,\theta_{\alpha,i})
  \sign{w_\alpha}w_{\alpha+ e_i} 
  - \sum_\alpha \partial_2F^i(v_{\alpha},\theta_{\alpha,i})
  |w_{\alpha+ e_i}|\Bigr]
\end{aligned}
 \end{equation*}
 Since each $F^i$ is monotone, it follows that $\term_1 \ge 0$. 
 Considering $\term_2$, we have
 \begin{align*}
\term_2 & = \frac{1}{\Dx^2}\sum_{i = 1}^d \sum_\alpha \Bigl[
a(\xi_{\alpha + e_i})\sign{w_\alpha}w_{\alpha+e_i}  
\\ & \qquad\qquad\qquad\qquad 
-2a(\xi_\alpha)|w_\alpha| 
+ a(\xi_{\alpha-e_i})\sign{w_\alpha}w_{\alpha-e_i}\Bigr],
 \end{align*}
 from which it follows that $\term_2 \le 0$.
\end{proof}

\begin{lemma}\label{lem:SemiDiscProp}
  Suppose $F^i$ is monotone for each $ 1 \le i \le d$. For any
  positive $T$, there
  exists a unique solution $u = \{u_\alpha\}$ to
  \eqref{eq:Semi-discreteScheme} on $[0,T]$ with the properties:
  \begin{itemize}
  \item[(i)] $\|u(t)\|_1 \le \|u_0\|_1$.
  \item[(ii)] For every $\alpha \in \mathbb{Z}^d$ and $t \in [0,T]$,
    \begin{equation*}
      \inf_\beta\{u_{\beta,0}\} \le u_\alpha(t) \le \sup_\beta\{u_{\beta,0}\}.
    \end{equation*}
  \item[(iii)] $|u(t)|_{BV} \le |u_0|_{BV}$.
  \item[(iv)] If $v = \{v_\alpha\}$ is a solution of the same problem
    with initial data $v_0$, then
    \begin{equation*}
      \|u(t) - v(t)\|_1 \le \|u_0-v_0\|_1.
    \end{equation*}
  \end{itemize}
\end{lemma}

\begin{proof}
  Parts (i),(iii) and (iv) follows since $\mathcal{S}(t)$ is a
  contraction semigroup. Part (ii) follows from
  \cite{ChambolleLucier1998}.
\end{proof}

Note that the $\ell^1$ bound in [(i)] implies that $u_\alpha(t)$
exists for all $t$, and not only for $t$ small.
\begin{lemma}\label{lem:semidisc_cont}
  Suppose $F^i$ is monotone for each $1 \le i \le d$. If $u$ is a
  solution to \eqref{eq:Semi-discreteScheme} and $\A(u_0) \in
  \ell^1(\Z^d)$, then for each $h > 0$,
  \begin{equation*}
    \|u(t + h)-u(t)\|_{\ell^1} \le \|\A(u_0)\|_{\ell^1} h.
  \end{equation*} 
\end{lemma}

\begin{proof}
Suppose that $\|u'(t)\| \le C$. Then
\begin{equation*}
	\|u(t+h)-u(t)\| = \left\|\int_t^{t+h}u'(s)\,ds\right\| \le \int_t^{t+h}\|u'(s)\|\,ds \le Ch,
\end{equation*}
and so Lipschitz continuity would follow. We claim that
\begin{equation}\label{eq:TimeDerDecrease}
\frac{\partial}{\partial t}\|u'(t)\| \le 0.
\end{equation}
Indeed,
\begin{equation*}
\begin{split}
 \frac{\partial}{\partial t}\|u'(t)\| &=
 \frac{\partial}{\partial t}\|\mathcal{A}(u(t))\| \\ 
 &=  \frac{\partial}{\partial t}\Bigl[\Dx^d\sum_{\alpha}
 \sign{\A(u(t))_\alpha}\A(u(t))_\alpha\Bigr]\\ 
 &= \Dx^d\sum_{\alpha}
 \sign{\A(u(t))_\alpha}\partial_t\A(u\alpha(t))_\alpha,
\end{split}
\end{equation*}
and
\begin{equation*}
\begin{split}
 & \partial_t\A(u(t))_\alpha 
 = \frac{\partial}{\partial t}\sum_{i = 1}^d  
 \Dm^i\left[F^i(u_\alpha(t),u_{\alpha +e_i}(t))
 -\Dp^iA(u_\alpha(t))\right] \\
 &\, = \sum_{i = 1}^d 
 \Dm^i\left[\partial_1F^i(u_\alpha(t),u_{\alpha
  +e_i}(t))u_\alpha'(t) + \partial_2F^i(u_\alpha(t),u_{\alpha
  +e_i}(t))u_{\alpha +e_i}'(t)\right] \\ 
 & \qquad \qquad -\sum_{i = 1}^d
 \Dm^i\Dp^ia(u_\alpha(t))u_\alpha'(t) \\ 
 &\, = -\sum_{i = 1}^d
 \Dm^i\left[\partial_1F^i(u_\alpha(t),u_{\alpha
  +e_i}(t))\A(u(t))_\alpha
+ \partial_2F^i(u_\alpha(t),u_{\alpha
  +e_i}(t))\A(u(t))_{\alpha +e_i}\right] \\ 
 & \qquad \qquad +\sum_{i=1}^d
 \Dm^i\Dp^ia(u_\alpha(t))\A(u(t))_\alpha.
\end{split}
\end{equation*}
Considering the similarity between this computation and
\eqref{eq:AccretiveComp}, it is seen 
that \eqref{eq:TimeDerDecrease} holds. We conclude 
that $\|u'(t)\| \le \|\A(u_0)\|$ and so the lemma follows.
\end{proof}


\begin{thebibliography}{10}

\bibitem{Andreianov:2009kx}
B.~Andreianov, M.~Bendahmane, and K.~H. Karlsen.
\newblock Discrete duality finite volume schemes for doubly nonlinear
  degenerate hyperbolic-parabolic equations.
\newblock {\em J. Hyperbolic Differ. Equ.}, 7(1):1--67, 2010.

\bibitem{Andreianov:2012fk}
B.~Andreianov and N.~Igbida.
\newblock On uniqueness techniques for degenerate convection-diffusion
  problems.
\newblock {\em Int. J. Dyn. Syst. Differ. Equ.}, 4(1-2):3--34, 2012.

\bibitem{Aregba-Driollet:2003}
D.~Aregba-Driollet, R.~Natalini, and S.~Tang.
\newblock Explicit diffusive kinetic schemes for nonlinear degenerate parabolic
  systems.
\newblock {\em Math. Comp.}, 73(245):63--94 (electronic), 2004.

\bibitem{Barles:2005sc}
G.~Barles and E.~R. Jakobsen.
\newblock Error bounds for monotone approximation schemes for
  {H}amilton-{J}acobi-{B}ellman equations.
\newblock {\em SIAM J. Numer. Anal.}, 43(2):540--558 (electronic), 2005.

\bibitem{Bouchut:2000dp}
F.~Bouchut, F.~R. Guarguaglini, and R.~Natalini.
\newblock Diffusive {BGK} approximations for nonlinear multidimensional
  parabolic equations.
\newblock {\em Indiana Univ. Math. J.}, 49(2):723--749, 2000.

\bibitem{Bouchut:1998ys}
F.~Bouchut and B.~Perthame.
\newblock Kru\v zkov's estimates for scalar conservation laws revisited.
\newblock {\em Trans. Amer. Math. Soc.}, 350(7):2847--2870, 1998.

\bibitem{Caffarelli:2008aa}
L.~A. Caffarelli and P.~E. Souganidis.
\newblock A rate of convergence for monotone finite difference approximations
  to fully nonlinear, uniformly elliptic {PDE}s.
\newblock {\em Comm. Pure Appl. Math.}, 61(1):1--17, 2008.

\bibitem{Carrillo:1999hq}
J.~Carrillo.
\newblock Entropy solutions for nonlinear degenerate problems.
\newblock {\em Arch. Ration. Mech. Anal.}, 147(4):269--361, 1999.

\bibitem{ChambolleLucier1998}
A.~Chambolle and B.~J. Lucier.
\newblock Un principe du maximum pour des op\'erateurs monotones.
\newblock {\em C. R. Acad. Sci. Paris S\'er. I Math.}, 326(7):823--827, 1998.

\bibitem{Chen:2005wf}
G.-Q. Chen and K.~H. Karlsen.
\newblock Quasilinear anisotropic degenerate parabolic equations with
  time-space dependent diffusion coefficients.
\newblock {\em Commun. Pure Appl. Anal.}, 4(2):241--266, 2005.

\bibitem{Chen:2006oy}
G.-Q. Chen and K.~H. Karlsen.
\newblock {$L^1$}-framework for continuous dependence and error estimates for
  quasilinear anisotropic degenerate parabolic equations.
\newblock {\em Trans. Amer. Math. Soc.}, 358(3):937--963, 2006.

\bibitem{ChenPerthame2003}
G.-Q. Chen and B.~Perthame.
\newblock Well-posedness for non-isotropic degenerate parabolic-hyperbolic
  equations.
\newblock {\em Ann. Inst. H. Poincar\'e Anal. Non Lin\'eaire}, 20(4):645--668,
  2003.

\bibitem{Cockburn:2003ys}
B.~Cockburn.
\newblock Continuous dependence and error estimation for viscosity methods.
\newblock {\em Acta Numer.}, 12:127--180, 2003.

\bibitem{CrandallLiggett1971}
M.~G. Crandall and T.~M. Liggett.
\newblock Generation of semi-groups of nonlinear transformations on general
  {B}anach spaces.
\newblock {\em Amer. J. Math.}, 93:265--298, 1971.

\bibitem{CranLions:FDM84}
M.~G. Crandall and P.-L. Lions.
\newblock Two approximations of solutions of {H}amilton-{J}acobi equations.
\newblock {\em Math. Comp.}, 43(167):1--19, 1984.

\bibitem{Dafermos:2010fk}
C.~M. Dafermos.
\newblock {\em Hyperbolic conservation laws in continuum physics}, volume 325
  of {\em Grundlehren der Mathematischen Wissenschaften [Fundamental Principles
  of Mathematical Sciences]}.
\newblock Springer-Verlag, Berlin, third edition, 2010.

\bibitem{Evje:2000ix}
S.~Evje and K.~H. Karlsen.
\newblock Discrete approximations of {BV} solutions to doubly nonlinear
  degenerate parabolic equations.
\newblock {\em Numer. Math.}, 86(3):377--417, 2000.

\bibitem{EvjeKarlsen2000}
S.~Evje and K.~H. Karlsen.
\newblock Monotone difference approximations of {BV} solutions to degenerate
  convection-diffusion equations.
\newblock {\em SIAM J. Numer. Anal.}, 37(6):1838--1860 (electronic), 2000.

\bibitem{EvjeKarlsen2002}
S.~Evje and K.~H. Karlsen.
\newblock An error estimate for viscous approximate solutions of degenerate
  parabolic equations.
\newblock {\em J. Nonlinear Math. Phys.}, 9(3):262--281, 2002.

\bibitem{Eymard:2002eu}
R.~Eymard, T.~Gallou{\"e}t, and R.~Herbin.
\newblock Error estimate for approximate solutions of a nonlinear
  convection-diffusion problem.
\newblock {\em Adv. Differential Equations}, 7(4):419--440, 2002.

\bibitem{Eymard:2002nx}
R.~Eymard, T.~Gallou{\"e}t, R.~Herbin, and A.~Michel.
\newblock Convergence of a finite volume scheme for nonlinear degenerate
  parabolic equations.
\newblock {\em Numer. Math.}, 92(1):41--82, 2002.

\bibitem{Holden:2010fk}
H.~Holden, K.~H. Karlsen, K.-A. Lie, and N.~H. Risebro.
\newblock {\em Splitting methods for partial differential equations with rough
  solutions}.
\newblock EMS Series of Lectures in Mathematics. European Mathematical Society
  (EMS), Z\"urich, 2010.
\newblock Analysis and MATLAB programs.

\bibitem{KKR2012}
K.~H. Karlsen, U.~Koley, and N.~H. Risebro.
\newblock An error estimate for the finite difference approximation to
  degenerate convection-diffusion equations.
\newblock {\em Numer. Math.}, 121(2):367--395, 2012.

\bibitem{Karlsen:2001ul}
K.~H. Karlsen and N.~H. Risebro.
\newblock Convergence of finite difference schemes for viscous and inviscid
  conservation laws with rough coefficients.
\newblock {\em M2AN Math. Model. Numer. Anal.}, 35(2):239--269, 2001.

\bibitem{KRS2014}
K.~H. Karlsen, N.~H. Risebro, and E.~B. Storr{\o}sten.
\newblock {$L^1$} error estimates for difference approximations of degenerate
  convection-diffusion equations.
\newblock {\em Math. Comp.}, 83(290):2717--2762, 2014.

\bibitem{Kruzkov:1970kx}
S.~N. Kru{\v{z}}kov.
\newblock First order quasilinear equations with several independent variables.
\newblock {\em Mat. Sb. (N.S.)}, 81 (123):228--255, 1970.

\bibitem{Krylov:2005lj}
N.~V. Krylov.
\newblock The rate of convergence of finite-difference approximations for
  {B}ellman equations with {L}ipschitz coefficients.
\newblock {\em Appl. Math. Optim.}, 52(3):365--399, 2005.

\bibitem{Kuznetsov}
N.~N. Kuznecov.
\newblock The accuracy of certain approximate methods for the computation of
  weak solutions of a first order quasilinear equation.
\newblock {\em \v Z. Vy\v cisl. Mat. i Mat. Fiz.}, 16(6):1489--1502, 1627,
  1976.

\bibitem{Lions:1994qy}
P.-L. Lions, B.~Perthame, and E.~Tadmor.
\newblock A kinetic formulation of multidimensional scalar conservation laws
  and related equations.
\newblock {\em J. Amer. Math. Soc.}, 7(1):169--191, 1994.

\bibitem{Makridakis:2003aa}
C.~Makridakis and B.~Perthame.
\newblock Optimal rate of convergence for anisotropic vanishing viscosity limit
  of a scalar balance law.
\newblock {\em SIAM J. Math. Anal.}, 34(6):1300--1307, 2003.

\bibitem{Ohlberger:2001oq}
M.~Ohlberger.
\newblock A posteriori error estimates for vertex centered finite volume
  approximations of convection-diffusion-reaction equations.
\newblock {\em M2AN Math. Model. Numer. Anal.}, 35(2):355--387, 2001.

\bibitem{Pavel1984}
N.~H. Pavel.
\newblock {\em Differential equations, flow invariance and applications},
  volume 113 of {\em Research Notes in Mathematics}.
\newblock Pitman (Advanced Publishing Program), Boston, MA, 1984.

\bibitem{Perthame1998}
B.~Perthame.
\newblock Uniqueness and error estimates in first order quasilinear
  conservation laws via the kinetic entropy defect measure.
\newblock {\em J. Math. Pures Appl. (9)}, 77(10):1055--1064, 1998.

\bibitem{Sato1968}
K.~Sato.
\newblock On the generators of non-negative contraction semigroups in {B}anach
  lattices.
\newblock {\em J. Math. Soc. Japan}, 20:423--436, 1968.

\bibitem{VolpertHudjaev1969}
A.~I. Vol{\cprime}pert and S.~I. Hudjaev.
\newblock The {C}auchy problem for second order quasilinear degenerate
  parabolic equations.
\newblock {\em Mat. Sb. (N.S.)}, 78 (120):374--396, 1969.

\bibitem{Wu:1989ve}
Z.~Q. Wu and J.~X. Yin.
\newblock Some properties of functions in {$BV_x$} and their applications to
  the uniqueness of solutions for degenerate quasilinear parabolic equations.
\newblock {\em Northeast. Math. J.}, 5(4):395--422, 1989.

\end{thebibliography}

\def\cprime{$'$}

\end{document}